\documentclass[a4paper]{article}

\usepackage{fullpage}
\usepackage[latin1]{inputenc}

\usepackage[affil-it]{authblk}

\usepackage{graphics}
\usepackage{epsfig}
\usepackage{amsmath}
\usepackage{amssymb}
\usepackage{fleqn}
\usepackage{enumitem}
\usepackage{epstopdf}
\usepackage[dvipsnames]{xcolor}

\usepackage{color}

\def\ul{\underline}

\def\ds{\displaystyle}
\def\proof{\noindent{\em Proof.} }
\def\qed{\vspace{2mm} \vrule height4pt width3pt depth2pt}

\newtheorem{theo}{Theorem}
\newtheorem{proposition}{Proposition} 
\newtheorem{corollary}{Corollary}
\newtheorem{lemma}{Lemma} 

\newtheorem{remark}{Remark}  

\newcommand{\R}{\mathbb{R}}
\newcommand{\eps}{\varepsilon}

\title{Dynamical modelling and optimal control of landfills}
\author{A. Rapaport}
\affil{MISTEA, UMR INRA/SupAgro, 2 pl. Viala, 34060 Montpellier,
  France

and MODEMIC, EPI INRA/Inria, 2004 rte des Lucioles, 06902 Sophia-Antipolis,
France}
\author{T. Bayen}
\affil{Universit\'e Montpellier 2, CC 051, 34095 Montpellier cedex 5,
  France

and MODEMIC, EPI INRA/Inria, 2004 rte des Lucioles, 06902 Sophia-Antipolis
France}
\author{M. Sebbah}
\affil{Univ. Tecnica Federico Santa Maria, Dep. Mat., Avda Espana 1680, Valparaiso, Chile}
\author{A. Donoso}
\affil{BIONATURE, CIRIC, INRIA-Chile. Escuela de Ingeniería
  Bioquímica, Pontificia Universidad Catolica de Valparaíso, General Cruz 34, Valparaíso, Chile}
\author{A. Torrico}
\affil{Centro de Modelamiento Matem\'atico, Universidad de Chile, , Beauchef 851, Santiago, Chile}
\date{\today}

\begin{document}

\maketitle

\begin{abstract}
We propose a simple model of landfill and study a minimal time
control problem where the re-circulation leachate is the manipulated variable.
We propose a scheme to construct the optimal strategy by dividing the state
space into three subsets $\mathcal{E}_0$, $\mathcal{Z}_{1}$
and the complementary. On $\mathcal{E}_0$ and
  $\mathcal{Z}_{1}$, the optimal control is constant until reaching
  target, while it can exhibit a singular arc outside these two
  subsets. Moreover, the singular arc could have a {\em barrier}.
In this case, we prove the existence of a switching curve 
that passes through a point of {\em prior saturation} under the assumption that 
the set $\mathcal{E}_0$ intersects the singular arc.
Numerical computations allow then to determine the switching curve and depict
the optimal synthesis.

\end{abstract}

\section{Introduction}

Landfills are controlled sites where the solid waste is disposed and it is slowly treated and stabilized under anaerobic conditions. Depending  on the specific region context this disposal method is highly encouraged (developing countries and big developed countries) or they are being replaced by more sustainable ways of waste treatment (small and medium-size developed countries). 

Landfill leachate is the liquid effluent generated during the landfill
operation. This waste-water is quite problematic due to its complex
composition thus the existing treatment technologies for this
waste-water are very costly.  During the first years of the landfill
operation the solubilization and fatty acid transformation of the
organic soluble compounds is mainly carried out, which means that
methane production is low \cite{H83}. Regardless, there are some key
factor that influenced the landfill behavior such as the re-circulation
leachate flow which increases the bio-reaction rates since it improved
the system mixing. Overall and due to scale reason (this bio-reactor is
humongous), the re-circulation flow may represent the only variable that can be at a
certain level manipulated and controlled once the landfill has begun
to operate. Mathematical models have been increasingly applied for
analysis, control and optimization of bio-processes. However few
application may be found in the literature in regards to control of
landfill operation. Due to the complexity of the system PDEs-based
models or Computational Fluid dynamic has been mainly used to
represent the process \cite{FB10,GT07,MR00}. A mechanistic model
assuming several considerations and all the steps in anaerobic
digestion was developed in \cite{SNBCB01} in which ordinary differential were used assuming perfect mixing.

When dealing with complex system such as anaerobic digestion it has
been observed that in some cases, using simplifies mechanistic
approaches may yield to results as good as the ones obtained using
over-parameterized models \cite{RRLB08}. Optimal
control strategies to estimate minimal time have already been shown to
be quite useful in order
to get an insight into the best performances of expect from piloting
efficiently bio-processes (see for instance \cite{BGM13,BRS14,GRR08,M99,RD11,RRGR14}).

In this work, we consider a simplified mathematical model of the dynamics
of solubilized and unsolubilized substrates to be bio-converted in a
landfill. The objective of the control problem is to drive the system as fast a possible to
low values of both concentrations of substrate, controlling the
leachate re-circulation.
We show that the optimal strategy is bang-bang with a possible
singular arc, but the determination of the optimal locus of switching is not
straightforward and requires a precise analysis. 
The number of switching times, and the on-line variables required to
be known or estimated for making the decision to switch at the right
time, depend on both the characteristics of the bacterial growth and on the
initial condition. 
Therefore this analysis provide new insights for the real-time
piloting of landfill, in terms of sensors, actuators and initial
conditions to be chosen by the practitioners that have to manipulate
the re-circulation flow.

The paper is organized as follows. In Section \ref{sec-statement}, we
introduce the optimal control problem, and we give properties on the
control system. 
In Section \ref{sec-opti1}, we state the Pontryagin
Maximum Principle and introduce a partition of the state space.
Section \ref{sec-constant} shows that on two subsets of initial conditions,
${\cal Z}_{1}$ and ${\cal E}_{0}$, the target is reached optimally with a constant control (see Propositions \ref{proposition1} and \ref{proposition2}).
Then, Section \ref{sec-admissible} gives the complete optimal synthesis when
there is no singular arc (see Proposition \ref{propositionMonod}) or when the singular arc is admissible which means that 
the singular control takes lower values than the upper bound $u_{max}$ for controls (see {Proposition \ref{propositionHaldane1}).
In Section \ref{sec-opti2}, we study the particular case where 
the singular control saturates the maximal admissible value. 
In this case, the singular arc has a {\it{barrier}} \cite{bosc} that corresponds 
 to the set of points of the singular arc where the singular control takes larger values than the upper bound. We prove the existence of a point 
of prior saturation and a switching curve $\mathcal{C}_1$ under the condition that $\mathcal{E}_0$ intersects the singular arc. 
This means that optimal trajectories should leave the singular locus before the saturation point (which is the unique 
point of the singular locus where the singular control equals
the maximal re-circulation flow. An optimal feedback of the problem is then given in Theorem \ref{main1}.
The assumption of non-emptiness is crucial in order to obtain the optimal synthesis in presence of a saturation point on the singular locus. 
In fact, an important feature of the system is that the boundary of the state space is invariant by the system. 
Therefore, if this intersection is empty, then the Pontryagin Maximum Principle does not allow us to 
exclude extremal trajectories with the constant control $u_{max}$ to be optimal until reaching $\mathcal{E}_0$ (see Theorem \ref{main1bis}).}
Section \ref{numeric-sec} depicts the optimal synthesis in the different cases appearing in the analysis of the problem. 
We end the paper by a conclusion with application perspectives.

\section{Model and preliminaries}{\label{sec-statement}

In the spirit of mathematical modeling in microbiology
  \cite{P75,SW95}, we propose a
  model of homogeneous landfill with a specific effect of
  a re-circulation flow on the bacterial activity, that is described by
  the following differential equations.
\begin{equation}
\left\{
\begin{array}{lll}\label{eq1a-eq1b-eq1c}
\dot S_{1}  & = &  -\gamma(Q)  f(S_{1}),\\
\dot S_{2} &  = &  \gamma(Q) f(S_{1})-\mu(S_{2})X,\\
\dot X  & = &  \mu(S_{2})X,
\end{array} 
\right.
\end{equation}
where $S_{1}$, $S_{2}$ stand respectively for unsolubilized and
solubilized substrates. $X$ is the concentration of the biomass that
degrades the solubilized substrate with a yield factor kept equal to
one (without any loss of generality, at the price to change the biomass
unit, one can always make this assumption) and specific growth rate
$\mu(\cdot)$. We assume that the reaction takes place in (closed) batch
conditions.
In addition, the re-circulation of the leachate, to be controlled with the
flow rate $Q \in [0,Q_{\max}]$, induces a solubilization of the
unsolubilized substrate $S_{1}$ into $S_{2}$ at a speed that depends
on $Q$, $S_{1}$ and possibly $X$, along with the following assumptions.
Following for instance \cite{SNBCB01}, we assume that $\gamma$ is
increasing over $\R_+$ with $\gamma(0)=0$. Therefore, we may set 

$$
u:=\frac{\gamma(Q)}{\gamma(Q_{\max})},
$$
which can be chosen as new control variable. Without any loss of generality, we can assume that $u$ is a measurable function w.r.t. the time $t$ 
taking values within $[0,1]$, i.e. the set of admissible controls is 
$$
\mathcal{U}:=\{u:[0,\infty) \rightarrow  [0,1] \; ; \; u \; \mathrm{meas.}\}.
$$  
We also require the following hypothesis on $f$:
\\

{\bf H0.} The function $f(\cdot)$ is increasing over $\R_+$ and satisfies $f(0)=0$.\\

We shall consider a general class of growth curves $\mu(\cdot)$, that
includes the usual Monod and Haldane ones:\\

{\bf H1.} The function $\mu(\cdot)$ is non-negative and equal to
zero only at $S_{2}=0$. Furthermore, there exists $S^{\star}_{2}>0$ such that
$\mu(\cdot)$ is increasing on $[0,S^{\star}_{2})$, decreasing on
$(S^{\star}_{2},+\infty)$, or $\mu(\cdot)$ is increasing and we put
$S^{\star}_{2}=+\infty$.\\

\bigskip 
One can straightforwardly check from \eqref{eq1a-eq1b-eq1c}
that the following property holds:
\[
\dot S_{1}+\dot S_{2}+\dot X=0 \quad \Rightarrow \quad
\exists M\geq 0 \mbox{ s.t. } S_{1}(t)+S_{2}(t)+X(t)=M, \quad \forall t \ .
\]
Given a positive value of the constant $M$ that characterizes a landfill, one
can rewrite the dynamics as a two-dimensional system
\begin{equation}\label{dynS1-dynS2}
\left\{
\begin{array}{lll}
\dot S_{1}  & = &  -uf(S_{1}),\\
\dot S_{2}  & = &  u f(S_{1})-\mu(S_{2})(M-S_{1}-S_{2}),
\end{array}
\right.
\end{equation} 
defined on the invariant domain
\[
{\cal D}:=\{(S_{1},S_{2})\in \R_{+}\times \R_+ \mbox{ with }
0<S_{1}+S_{2}< M\}.
\]
An important feature of the system is that the boundary sub-sets $\{0\} \times [0,M]$ and 
$N:=\{(S_1,S_2) \in \R_+\times \R_+ \; ; \; S_1+S_2=M\}$ are invariant by \eqref{dynS1-dynS2} 
(this property has several consequences on the optimal synthesis, see section \ref{pl}).}

The optimal control problem can be stated as follows. Given an initial condition in ${\cal D}$, the objective
is to drive in minimal time the state $S(\cdot)=(S_1(\cdot),S_2(\cdot))$ to a
target ${\cal T}$ for which $S_{1}$ and $S_{2}$ are below given positive thresholds $\ul
S_{1}$, $\ul S_{2}$:
\[
{\cal T}:=\{(S_{1},S_{2})\in [0,\ul S_{1}]\times [0,\ul
S_{2}] \},
\]
with $(\ul S_{1},\ul S_{2})\in {\cal D}$.\\

\bigskip

Let us first study the attainability of the target from any initial
condition in ${\cal D}$.

\begin{proposition} Given an initial condition $S^{0} \in {\cal
    D}\setminus {\cal T}$, the feedback law
\[
u[S]:=\left|
\begin{array}{ll} 
1 & \mbox{if } S_{1}>\ul S_{1},\\
0 & \mbox{otherwise},
\end{array}
\right.
\]
drives the state in finite time in ${\cal T}$.

\end{proposition}

\proof
Consider trajectories generated with the proposed feedback law.
If $S_{1}(0)>\ul S_{1}$, $S_{1}(\cdot)$ is solution of
\[
\dot S_{1} = -f(S_{1}),
\]
until $S_{1}(\cdot)$ reaches $\ul S_{1}$ in a finite time $T$,
the right member of the differential equation being strictly
negative. 
If $S_{1}(0)\leq \ul S_{1}$ we simply take $T=0$.
At time $T$, if $S_{2}(T)\leq \ul S_{2}$, the state is in the
target.
Otherwise, from time $T$, $S_{1}(t)$ stays equal
to $ S_{1}(T)$ for any future time $t$, and $S_{2}(\cdot)$ 
is solution of the differential equation
\[
\dot S_{2} = -\mu(S_{2})(M-S_{1}(T)-S_{2}).
\]
Consequently, $S_{2}(\cdot)$ is decreasing and therefore one has $S_2(t)\rightarrow 0$ when 
$t$ goes to $+\infty$ (this follows from the definition of $\mathcal{D}$, the monotonicity of $S_2(\cdot)$ for $t\geq T$, and the fact that $\mu(0)=0$).
Thus, the solution reaches $\ul S_{2}$ in finite
time, that is the state enters the target.
\qed

\bigskip

So the minimal time problem is well defined in ${\cal D}$.

\bigskip

We shall denote in the following
$S^{u}(\cdot)$, resp. 
$S^{max}(\cdot)$, a solution of 
\eqref{dynS1-dynS2} in the domain ${\cal
    D}\setminus{\cal T}$ for the control $u(\cdot)$ resp. the constant
control $u=1$.

\begin{lemma}
\label{Lemma1}
For any $S^{0}\in {\cal D}\setminus {\cal T}$ and control $u(\cdot)$, the
solutions $S^{u}(\cdot)$, $S^{max}(\cdot)$ with $S^{u}(0)=S^{max}(0)=S^{0}$ fulfill
\[
S^{u}(t)\in {\cal S}^{\max}:=\bigcup_{\tau \geq 0}\left\{S \in {\cal
  D}\setminus{\cal T} \, ; \, S_{1}=S_{1}^{\max}(\tau), \; S_{2}\leq
S_{2}^{\max}(\tau)\right\}, \quad \forall t\geq 0.
\]
\end{lemma}

\proof
Let $\tau \geq 0$ be given, and consider the point $(S_1^{\max}(\tau),S_2^{\max}(\tau))\in \mathcal{S}^{\max}$. Then the cross product of 
$(\dot{S}_1^{\max}(\tau),\dot{S}_2^{\max}(\tau))$ with $(S_1^{u}(\tau),S_2^{u}(\tau))$ at $(S_1^{\max}(\tau),S_2^{\max}(\tau))$ satisfies:
$$
(\dot{S}_1^{\max}(\tau),\dot{S}_2^{\max}(\tau)) \wedge (S_1^{u}(\tau),S_2^{u}(\tau)) =
\mu(S_{2})(M-S_{1}-S_{2})f(S_{1})(1-u) \geq 0.
$$
Moreover, at a given point on the segment $\{S_1^{\max}(0)\}\times [0,S_2^{\max}(0))$, one has $\dot{S}^u_1\leq 0$ for any control $u(\cdot)$.
Therefore, a trajectory cannot leave the set $\mathcal{S}^{\max}$ on its
boundary $(S_1^{\max}(\cdot),S_2^{\max}(\cdot))$ and $\{S_1^{\max}(0)\}\times [0,S_2^{\max}(0))$.
\qed


\section{Pontryagin's Principle and domain partition}
\label{sec-opti1}
We use the Pontryagin Maximum Principle \cite{PBGM64} in order to derive necessary conditions on optimal trajectories. 
The Hamiltonian $H=H(S_{1},S_{2},\lambda_{0},\lambda_{1},\lambda_{2},Q)$ associated to the control system is defined as :
\begin{equation}
\label{Hamiltonian}
H(S_{1},S_{2},\lambda_{0},\lambda_{1},\lambda_{2},Q):=\lambda_{0}
+u(\lambda_{2}-\lambda_{1})f(S_{1})-\lambda_{2}\mu(S_{2})(M-S_{1}-S_{2}).
\end{equation} 
The Pontryagin Maximum Principle can be stated as follows. Let $u(\cdot)$ an optimal control steering a point $(S_1^0,S_2^0)$ 
in minimal time to the target, and $S=(S_1,S_2)$ the associated trajectory. 
Then, there exists $t_f>0$, $\lambda_0\geq 0$, and an absolutely continuous map 
$\lambda=(\lambda_1,\lambda_2):[0,t_f] \rightarrow \R^2$ such that $(\lambda_0,\lambda_1(\cdot),\lambda_2(\cdot))\not=0$ 
and :
\begin{equation}{\label{adjoint1-adjoint2}}
\begin{cases}
\dot \lambda_{1}  =  -\partial H/\partial S_1 =
-u(\lambda_{2}-\lambda_{1})f^{\prime}(S_{1})-\lambda_{2}\mu(S_{2}),\\
\dot \lambda_{2}  =  -\partial H/\partial S_2 = \lambda_{2}(\mu^{\prime}(S_{2})(M-S_{1}-S_{2})-\mu(S_{2})).
\end{cases}
\end{equation}
for a.e. $t\in [0,t_f]$. Moreover, the Hamiltonian is minimized w.r.t. the control $u$ which means: 
\begin{equation}{\label{PMP}}
u(t) \in \mathrm{arg} \ \mathrm{min}_{\alpha \in [0,1]}H(S_{1}(t),S_{2}(t),\lambda_{0},\lambda_{1}(t),\lambda_{2}(t),\alpha), \; \; \mathrm{a.e.} \; t\in [0,t_f].
\end{equation}
We call extremal trajectory a triple $(S(\cdot),\lambda(\cdot),u(\cdot))$ satisfying \eqref{dynS1-dynS2}-\eqref{adjoint1-adjoint2}-\eqref{PMP}.
When $\lambda_0=0$, then we say that an extremal is {\it{abnormal}} whereas if $\lambda_0\not=0$, then we say that an extremal is {\it{normal}}.
Abnormal trajectories are studied in Corollary \ref{abnormal-coro}.
As $t_f$ is free, $H$ is equal to zero along any extremal trajectory.
Taking into account the geometry of the target set $\mathcal{T}$, we obtain the {\it{transversality conditions}}:
\begin{equation}
\label{transversal}
\lambda(t_{f})=\left|\begin{array}{ll}
(1,0) & \mbox{ if } S_{1}(t_{f})=\ul S_{1} \mbox{ and } 
S_{2}(t_{f})<\ul S_{2},\\
(\alpha,(1-\alpha)) & \mbox{ if } S_{1}(t_{f})=\ul S_{1} \mbox{ and } 
S_{2}(t_{f})=\ul S_{2}, \qquad (\mbox{with } \alpha \in [0,1]),\\
(0,1) & \mbox{ if } S_{1}(t_{f})<\ul S_{1} \mbox{ and } 
S_{2}(t_{f})=\ul S_{2}.
\end{array}\right.
\end{equation}
The switching function $\phi$ defined as $\phi:=\lambda_{2}-\lambda_{1}$ provides the control law.
An optimal control $u$ satisfies:
$$
\begin{cases}
\begin{array}{lll}
\phi(t)>0 & \Rightarrow & u(t)=0,\\ 
\hspace{-0.15cm}\phi(t)=0 & \Rightarrow & u(t)\in [0,1],\\
\phi(t)<0 & \Rightarrow & u(t)=1.
\end{array}
\end{cases}
$$
We say that a time $t_0\in [0,t_f]$ is a {\it{switching point}} (or switching time) if the control $u$ is non-constant in any neighborhood of $t_0$. 
In this case, one has $\phi(t_0)=0$, and we say that $\phi$ switches at time $t_0$. 
We say that an extremal trajectory has a {\it{singular arc}} if there exists a time interval $I:=[t_1,t_2] \subset [0,t_f]$ such that we have  
$\phi(t)=0$ for any time $t \in I$ (see \cite{BC02}). We then have $\phi=\dot{\phi}=0$ on $I$.
Moreover, one can easily check that
$\phi=0$ implies
$\dot\phi=\lambda_{2}\mu^{\prime}(S_{2})(M-S_{1}-S_{2})=0$. 
As $\phi$ and $\lambda_{2}$  cannot be equal to zero
simultaneously, we must have $\mu'(S_2)=0$ along the singular arc. Therefore, the singular locus is defined as the set
$$
\Delta:=(0,M-S_2^{\star})\times \{S^{\star}_{2}\},
$$ 
If $S^{\star}_{2}\geq M$, then the singular arc no longer exists.
When $S^{\star}_{2}< M$, we define the {\em singular} feedback control as:
\begin{equation}
\label{singular_feedback}
u_s(S_{1}):=\frac{\mu(S_{2}^{\star})(M-S_{1}-S_{2}^{\star})}{f(S_{1})}
, \; S_{1}>0 \ .
\end{equation}
Under Assumption {\bf{H0}}, the map $S_{1}\mapsto u_s(S_{1})$ is
decreasing with $u_s(0^{+})=+\infty$. The function
\begin{equation}
\label{nu}
\nu(S_{1}):=f(S_{1})-\mu(S_{2}^{\star})(M-S_{1}-S_{2}^{\star}),
\end{equation}
being increasing and such
that $\nu(0)<0$ and $\nu(M-S_{2}^{\star})>0$, one can then define
$S_{1}^{\min}$ as the unique root of $\nu(\cdot)$ on the interval
$(0,M-S_{2}^{\star})$. The number $S_{1}^{\min}$ defines the left
point limit of the admissible subset of the singular arc
$S_{2}=S_{2}^{\star}$ as one can easily check from (\ref{dynS1-dynS2})
that the following property is fulfilled
\[
\begin{array}{lll}
S_{1}\geq S_{1}^{\min}, \; S_{2}=S_{2}^{\star} & \Rightarrow &
u_s(S_{1})\leq 1 \mbox{ and } \dot S_{2}=0 \mbox{ for }
u=u_s(S_{1}), \  \\
S_{1} < S_{1}^{\min}, \; S_{2}=S_{2}^{\star} & \Rightarrow & \dot
S_{2}<0, \; \forall u \in [0,1]. \ 
\end{array}
\]

The point $(S_1^{\min},S_2^{\star})$ is called {\it{saturation point}}. Following \cite{bosc}, the part of the singular arc where $u_s$ is strictly larger than the maximal admissible value is called {\it{barrier}}, i.e. the 
singular control saturates the value $u=1$. This phenomena and its consequence on the optimal synthesis 
are studied precisely in Section \ref{sec-opti2}.

\bigskip

It is convenient for the characterization of the optimal synthesis to consider the following partition of the domain ${\cal D}\setminus {\cal T}$:
\[
{\cal Z}_{0}:=\left\{ (S_{1},S_{2})\in {\cal D}\setminus {\cal T} \; ; \;
S_{1}\leq \ul S_{1}, \; S_{2}>\ul S_{2} \right\},
\]
\[
{\cal Z}_{1}:=\left\{ (S_{1},S_{2})\in {\cal D}\setminus {\cal T} \; ; \;
S_{1} \in (\ul S_{1},\ul \sigma_{1}], \; S_{2}\leq \sigma_{2}(S_{1}) \right\},
\]
\[
{\cal Z}_{s}:=({\cal D}\setminus {\cal T})\setminus({\cal Z}_{0}\cup{\cal Z}_{1}),
\]
where $\sigma_{2}(\cdot)$ is solution of the differential equation
\[
\frac{d\sigma_{2}}{d\sigma_{1}}=
\frac{\mu(\sigma_{2})(M-\sigma_{1}-\sigma_{2})}{f(\sigma_{1})}-1
\]
for the Cauchy problem with initial condition $\sigma_{2}(\ul
S_{1})=\ul S_{2}$ on the interval $[\ul S_{1},\ul \sigma_{1}]$, where
$\ul \sigma_{1}$ is the smallest $\sigma_{1}>\ul S_{1}$ such that
$\sigma_{2}(\ul \sigma_{1})=0$.\\

\begin{remark}
It is worth pointing out that $\sigma_1\longmapsto \sigma_2(\sigma_1)$ 
is the unique solution of \eqref{dynS1-dynS2} backward in time 
starting at $(\ul S_1,\ul S_2)$ with the control $u=1$.
Since the set $N$ is invariant and $(\ul S_1, \ul S_2) \in \mathcal{D}$, this trajectory necessarily intersects the line segment $(0,M)\times \{0\}$.
\end{remark}

\begin{figure}[h]
\begin{center}
\includegraphics[scale=0.53]{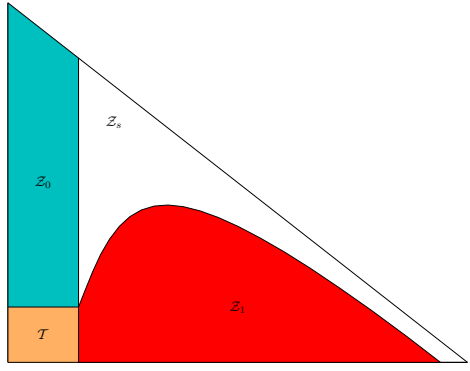}
\caption{Illustration of the subsets ${\cal T}$, $\mathcal{Z}_0$, $\mathcal{Z}_1$ and
  $\mathcal{Z}_s$. 
   \label{fig_Z}}
\end{center}
\end{figure}

\section{Characterization of optimal solutions with constant control
  and corollaries}
\label{sec-constant}

We first characterize optimal trajectories in the subset ${\cal Z}_{1}$.

\begin{proposition} 
\label{proposition1}
Assume that Hypotheses {\bf{H0}} and {\bf{H1}} are fulfilled.
For any initial condition in ${\cal Z}_{1}$, the optimal
trajectory stays in ${\cal Z}_{1}$ and the constant control $u=1$
is optimal control until reaching the target.
\end{proposition}

\proof
From any initial condition $S^{0}\in {\cal D}\setminus {\cal T}$ that
belongs to the graph of the function $\sigma_{2}(\cdot)$, and any control
$u(\cdot)$, Lemma \ref{Lemma1} shows that the trajectory stays in the
set ${\cal S}^{\max}$, whose boundary is contained in the boundary of 
${\cal Z}_{1}$.
So we conclude that the domain ${\cal Z}_{1}$ is invariant for any
control $u(\cdot)$. Consequently, any optimal trajectory in this domain
reaches the target at a time $t_{f}$ such that
$S_{1}(t_{f})=\ul S_{1}$ and $S_{2}(t_{f})\leq\ul S_{2}$.

For initial conditions such that $S_{2}(0)<\sigma_{2}(S_{1}(0))$, the
target is reached at states such that  $S_{2}(t_{f})<\ul S_{2}$.
Then, the transversality conditions (\ref{transversal}) give $\lambda_{2}(t_{f})=0$.
This last equality implies with the dynamics
(\ref{adjoint1-adjoint2}) that the variable $\lambda_{2}(t)$ is equal to $0$ for any time $t\in
[0,t_{f}]$. Then, the switching function is such that $\dot
\phi=-uf^{\prime}(S_{1})\phi \geq 0$ with $\phi(t_{f})=-1$. So $\phi$ stays
negative for any time and the optimal control is  constant
equal to $1$.

For initial conditions such that $S_{2}(0)=\sigma_{2}(S_{1}(0))$, the
optimal trajectory has to reach the target with $S_{2}(t_{f})=\ul
S_{2}$, and thus the constant control $u=1$ is optimal. 
Otherwise, as for the previous case, one should have
$\lambda_{2}(t)=0$ and $\phi(t)<0$ for any $t$, that is the optimal
control has to be constant equal to $1$, that implies
$S_{2}(t_{f})=\ul S_{2}$.
\qed\\

One can then formulate the following corollaries concerning the optimal
trajectories that lie outside the set ${\cal Z}_{1}$.

\begin{corollary}
\label{corollary2}{\label{abnormal-coro}}
From any initial condition in $({\cal D}\setminus {\cal T})\setminus
{\cal Z}_{1}$, an optimal trajectory reaches the target at a time
$t_{f}$ such that $S_{2}(t_{f})=\ul S_{2}$, and the constant control
$u=1$ cannot be optimal. Furthermore, the variable
$\lambda_{2}$ is positive along such optimal trajectories, and $\lambda_{0}>0$.
\end{corollary}

\proof
If an optimal trajectory reaches the target with
$S_{2}(t_{f})<\ul S_{2}$ from an initial condition
outside the set ${\cal Z}_{1}$, it has to cross the graph of the
function $\sigma_{2}(\cdot)$ before reaching the target. 
According to Proposition \ref{proposition1}, the optimal trajectory from this
boundary of ${\cal Z}_{1}$ reaches the target at the corner state $(\ul S_{1},\ul
S_{2})$, that is such that $S_{2}(t_{f})=\ul S_{2}$.

By uniqueness of the solution of the Cauchy problem with the constant
control $1$, a trajectory with an initial condition
outside the set ${\cal Z}_{1}$ cannot cross the graph of the function
$\sigma_{2}(\cdot)$ and consequently one has $S_{1}(t_{f})<\ul S_{1}$
when it reaches the target. Then from the transversality condition gives
$\phi(t_{f})>0$ and $u=1$ cannot be optimal.

From equation (\ref{adjoint1-adjoint2}), one can see that $\lambda_{2}$ is
either identically equal to zero or has constant sign. If
$\lambda_{2}$ is equal to zero, one should have $\lambda_{1}(t_{f})>0$
from conditions (\ref{transversal}), and the optimal trajectory has to
reach the target with $u=1$, that is not possible from
the above argumentation. 

Finally, the transversality condition (\ref{transversal}) provides the positive
sign of $\lambda_{2}$. Having $\lambda_{0}=0$ would imply, from
$H=0$ along any optimal trajectory (where the Hamiltonian $H$ is
given in (\ref{Hamiltonian})), to have the equality
\[
u\phi f(S_{1})=\lambda_{2}\mu(S_{2})(M-S_{1}-S_{2}),
\]
fulfilled at any time, that is to $u>0$ and $\phi>0$ at
any time. This is impossible for reaching the target.\qed

\begin{corollary}
\label{corollary3}
Consider an optimal trajectory and a time $t_0$ such that
$(S_{1}(t_0),S_{2}(t_0))\notin {\cal T}\cup{\cal Z}_{1}$. Then, the following properties are fulfilled:
\begin{enumerate}[label=\roman{*}.]
\item If $\phi(t_0)\geq 0$ with $S_{2}(t_0)<S_{2}^{\star}$, then we have 
$\phi(t)>0$ for any time $t>t_0$ such that $(S_{1}(t),S_{2}(t))\notin{\cal T}\cup{\cal Z}_{1}$. 
\item If $\phi(t_0)\leq 0$ with $S_{2}(t_0)>S_{2}^{\star}$, then we have 
$\phi(t)<0$ for any time $t>t_0$ such that $(S_{1}(t),S_{2}(t))\notin {\cal T}$ and
$S_{2}(t)\geq S_{2}^{\star}$.
\end{enumerate}
\end{corollary}

\proof
When the switching function $\phi$ is equal to zero, one has
\[
\dot \phi = \lambda_{2}\mu^{\prime}(S_{2})(M-S_{1}-S_{2}).
\]
As for any initial condition in  $({\cal D}\setminus {\cal T})\setminus
{\cal Z}_{1}$, we know from Corollary \ref{corollary2} that
$\lambda_{2}$ stays positive, we deduce the property that will be useful in the following:
\[
\phi=0 \Rightarrow \left|\begin{array}{ll}
\dot \phi >0 & \mbox{when } S_{2}<S_{2}^{\star},\\
\dot \phi <0 & \mbox{when } S_{2}>S_{2}^{\star},
\end{array}\right.
\]
Let us prove i . Suppose that $\phi(t_0)\geq 0$ with $S_{2}(t_0)<S_{2}^{\star}$. If there exists a time 
 $t>t_0$ such that $\phi(t)=0$ and $(S_{1}(t),S_{2}(t))\notin{\cal T}\cup{\cal Z}_{1}$. 
 We then have $\dot{\phi}(t)\leq 0$ and $u=0$ on $[t_0,t]$, therefore $S_2(t)< S_2^{\star}$, 
and $\dot{\phi}(t)>0$ (by the remark above). We then have a contradiction which proves i. 
The proof of ii follows in the same way. \qed

\bigskip

We  now study optimal trajectories in the domain ${\cal Z}_{0}$.
For this purpose, we consider the (possibly empty) subset of ${\cal
  Z}_{0}$ defined as
\begin{equation}
{\cal C}_{0}:=\left\{ (S_{1},S_{2}) \in {\cal Z}_{0} \; \vert \;
    \varphi(S_{1},S_{2})=1 \right\},
\end{equation}
where the function $\varphi(\cdot)$ is defined as follows.
\begin{equation}
\varphi(S_{1},S_{2})=\int_{\ul
  S_{2}}^{S_{2}}\frac{\mu^{\prime}(s)\mu(\ul S_{2})(M-S_{1}-\ul S_{2})}{\mu(s)^{2}(M-S_{1}-s)}ds.
\end{equation}
When $S^{\star}_{2} <M$ we shall also consider the {\em end singular state}
$(S_{1}^{\star},S^{\star}_{2})$ where $S_{1}^{\star}$ is defined as follows.
\begin{equation}
\label{end_singular}
S_{1}^{\star}=\left|\begin{array}{ll}
\ul S_{1} & \mbox{when } S^{\star}_{2}\geq\ul S_{2} \mbox{ and } {\cal
  C}_{0}=\emptyset,\\
\inf\{ S_{1}>0 \,
  \vert \, \varphi(S_{1},S^{\star}_{2})>1\}  & \mbox{when } \varphi(\ul
  S_{1},S^{\star}_{2})\geq 1,\\
\inf\{ S_{1}>\ul S_{1} \,
  \vert \, \sigma_{2}(S_{1})<S^{\star}_{2}\} & \mbox{when } S^{\star}_{2}< \ul S_{2}.
\end{array}\right.
\end{equation}

The different possible positions of $S_{1}^{\star}$ are illustrated in
  Section \ref{numeric-sec}. Let us now give some properties of the set ${\cal C}_{0}$.
\begin{lemma}
\label{Lemma2}
Assume Hypothesis {\bf{H1}}. The following cases occur depending on the
relative position of $S^{\star}_{2}$ w.r.t. to $M$ and $\ul S_{2}$.
\begin{enumerate}[label=\roman{*}.]
\item If $S^{\star}_{2}\geq M$, then ${\cal C}_{0}$ is the graph of a decreasing
  $C^{1}$ function $S_{1}\mapsto S_{2}^{c}(S_{1})$ defined on $[0,\ul
  S_{1}]$.
\item If $S^{\star}_{2} \in (\ul S_{2},M)$ and $\varphi(\ul
  S_{1},S^{\star}_{2})< 1$, then ${\cal C}_{0}$ is empty.
\item If $S^{\star}_{2} \in (\ul S_{2},M)$ and $\varphi(\ul
  S_{1},S^{\star}_{2})\geq 1$, ${\cal C}_{0}$ is the graph of a decreasing
  function $S_{1}\mapsto S_{2}^{c}(S_{1})$ defined on
  $[S_{1}^{\star},\ul S_{1}]$ that is $C^{1}$ on $(S_{1}^{\star},\ul
  S_{1}]$. Furthermore, one has
   $S_{2}^{c}(S_{1})<S^{\star}_{2}$ for any $S_{1}\in
   (S_{1}^{\star},\ul S_{1}]$.
   When $S_{1}^{\star}>0$, one has $S_{2}^{c}(S_{1}^{\star})=S_{2}^{\star}$.
   The graph of the function $S_{2}^{c}(\cdot)$ has a vertical slope
   at $S_{1}^{\star}$ when $S_{1}^{\star}>0$ or $S_{1}^{\star}=0$ with
   $S_{2}^{c}(0)=S_{2}^{\star}$.
\item When $S^{\star}_{2}\leq \ul S_{2}$, then  ${\cal C}_{0}$ is
  empty and one has necessarily $S_{1}^{\min}\leq S_{1}^{\star}$.
\end{enumerate}
\end{lemma}

\proof
Consider the case $S^{\star}_{2}\geq M$. For each $S_{1}\in [0,\ul S_{1}]$,
the map $S_{2}\mapsto \varphi(S_{1},S_{2})$ is increasing and there
exists a
number $m>0$ such that 
$\mu^{\prime}(S_{2})/\mu(S_{2})^{2}\geq m$ for any $S_{2}\in[\ul
S_{2},M-S_{1}]$.
Then, one can write
$$
\varphi(S_{1},S_{2})\geq m\mu(\ul S_{2})(M-S_{1}-\ul S_{2})\int_{\ul
  S_{2}}^{S_{2}}\frac{ds}{M-S_{1}-s}=m\mu(\ul S_{2})(M-S_{1}-\ul
  S_{2})\ln\left(\frac{M-S_{1}-S_{2}}{M-S_{1}-\ul S_{2}}\right),
$$
and deduce $\lim_{S_{2}\to
  M-S_{1}}\varphi(S_{1},S_{2})=+\infty$. Consequently, for each 
$S_{1}\in [0,\ul S_{1}]$ there exists an unique $S^{c}_{2}> \ul
S_{2}$ such that $\varphi(S_{1},S^{c}_{2})=1$. Furthermore, one has
for any $S_{2}>\ul S_{2}$
\begin{equation}
\label{derivatives}
\frac{\partial \varphi}{\partial S_{1}}(S_{1},S_{2}) =
\int_{\ul S_{2}}^{S_{2}}\frac{\mu^{\prime}(s)\mu(\ul
  S_{2})(S_{2}-\ul
  S_{2})}{\mu(s)^{2}(M-S_{1}-s)^{2}}ds, \;
\frac{\partial \varphi}{\partial S_{2}}(S_{1},S_{2}) =
\frac{\mu^{\prime}(S_{2})\mu(\ul S_{2})(M-S_{1}-\ul
  S_{2})}{\mu(S_{2})^{2}(M-S_{1}-S_{2})},
\end{equation}
that are both positive, 
and by the Implicit Function Theorem we conclude that $S_{1}\mapsto
S_{2}^{c}(S_{1})$ is a $C^{1}$ decreasing map defined over $[0,\ul S_{1}]$. This proves i.

Let us now prove ii. and iii. Consider the case $S^{\star}_{2}< M$. When $S^{\star}_{2}\leq \ul S_{2}$,
${\cal C}_{0}$ is clearly empty. When $S^{\star}_{2}> \ul S_{2}$, let us take $S_{1}\in [0,\ul S_{1}]$.
Clearly the map $S_{2}\mapsto \varphi(S_{1},S_{2})$ is
non increasing for $S_{2}\geq S^{\star}_{2}$. So $S_{2}^{c}$, if it exists,
has to be less or equal to $S^{\star}_{2}$.
One can observe the following facts:
\begin{itemize}
\item[-] the map $S_{2}\mapsto \varphi(S_{1},S_{2})$ is
increasing on $[\ul S_{2},S^{\star}_{2}]$ for any $S_{1}\in [0,\ul S_{1}]$,
\item[-] the map $S_{1}\mapsto \varphi(S_{1},S_{2})$ is
increasing on $[0,\ul S_{1}]$ for any $S_{2}\in (\ul S_{2},S^{\star}_{2}]$
(see the derivative (\ref{derivatives})),
\end{itemize}
and deduce that when $\varphi(\underline S_{1},S^{\star}_{2})< 1$, then
one has $\varphi(S_{1},S_{2})<1$ for any $(S_{1},S_{2})\in [0,\ul
S_{1}]\times (\ul S_{2},S^{\star}_{2}]$. The set ${\cal C}_{0}$ is then
empty in this case. Otherwise, for any $S_{1}>S_{1}^{\star}$, there exists an unique
$S_{2}^{c}\in(\ul S_{2},S^{\star}_{2})$ such that $\varphi(S_{1},S_{2}^{c})=1$.
As previously, one can use the 
Implicit Function Theorem to write
\begin{equation}
\label{dS2c}
S_{2}^{c\,\prime}(S_{1})=-\frac{\ds \frac{\partial\varphi}{\partial
    S_{1}}(S_{1},S_{2}^{c}(S_{1}))}{\ds \frac{\partial\varphi}{\partial
    S_{2}}(S_{1},S_{2}^{c}(S_{1}))}<0, \quad \forall S_{1} \in
(S_{1}^{\star},\ul S_{1}].
\end{equation}
and thus conclude that $S{1}\mapsto
S_{2}^{c}(S_{1})$ is a decreasing map defined on
$[S_{1}^{\star},\ul S_{1}]$ and of class $C^{1}$ over $(S_{1}^{\star},\ul S_{1}]$.
If $S_{1}^{\star}>0$ then one has necessarily
$S_{2}^{c}(S_{1}^{\star})=S_{2}^{\star}$. Otherwise, one should have
$\varphi(S_{1}^{\star},S_{2}^{\star})>1$ and by continuity of
$\varphi$, there should exist a neighborhood of
$(S_{1}^{\star},S_{2}^{\star})$ with $\varphi$ larger that $1$, in
contradiction with the definition of $S_{1}^{\star}$. Finally, from
(\ref{dS2c}) and (\ref{derivatives}), one has:
\[
\lim_{S_{1}\to
  S_{1}^{\star},S_{1}>S_{1}^{\star}}S_{2}^{c\,\prime}(S_{1})=-\infty 
\]
when $S_{2}^{c}(S_{1}^{\star})=S_{2}^{\star}$.\\

We end the proof showing that $S_{1}^{\min}\leq S_{1}^{\star}$
when $S_{2}^{\star}\leq \underline S_{2}$. If not, one has:
\[
S_{1} \leq S_{1}^{\star}, \, S_{2}=S_{2}^{\star}
\Rightarrow \dot S_{2} <0,
\]
for any solution $(S_1(\cdot),S_2(\cdot))$ of \eqref{dynS1-dynS2}.
Then the trajectory $(S_{1}(\cdot),S_{2}(\cdot))$ with the initial condition
$(S_{1}^{\star},S_{2}^{\star})$ and the constant
control $u=1$ is such that $S_{2}(t)<S_{2}^{\star}$ for any
$t>0$. But one has $S_{2}^{\star}=\sigma_{2}(S_{1}^{\star})$ and 
this trajectory verifies also $S_{2}(t)=\sigma_{2}(S_{1}(t))$
for any $t>0$ such that $S_{1}(t) \in [\underline
S_{1},S_{1}^{\star})$, and thus $\sigma_{2}(\underline
S_{1})<S_{2}^{\star}\leq \underline S_{2}$, which is in contradiction
with the definition of $\sigma_{2}(\cdot)$. This proves iv.
\qed

\bigskip

We can first characterize the optimal trajectories on the subset
\[
{\cal E}_{0}:=\overline{\left\{ (S_{1},S_{2}) \in {\cal Z}_{0} \; \vert \; 
(S_{1},s_{2})\notin {\cal C}_{0} , \; \forall s_{2}< S_{2} \right\}} \
.
\]

\begin{proposition}
\label{proposition2}
Assume that Hypotheses {\bf{H0}} and {\bf{H1}} are fulfilled.
\begin{enumerate}[label=\roman{*}.]
\item For any state in ${\cal E}_{0}$,
the constant control $u=0$ is optimal until reaching the target,
\item The switching function $\phi$ is
equal to zero at any state in ${\cal C}_{0}$.
\end{enumerate}
\end{proposition}

\proof
Consider an initial condition $S^{0}$ in ${\cal Z}_{0}$ that is such that
$S_{1}^{0}\leq \ul S_{1}$.
From equation (\ref{dynS1-dynS2}) one has $\dot S_{1}\leq 0$ whatever is the
control $u(\cdot)$, and then $S_{1}(t)\leq \ul S_{1}$ for any positive
time. Consequently any trajectory stays in the set ${\cal Z}_{0}$ until reaching
the target.

When $S_{1}^{0}< \ul S_{1}$, an optimal trajectory has to reach the
target at a time $t_{f}>0$ such that $S_{1}(t_{f})<\ul S_{1}$. From the
transversality condition (\ref{transversal}) one has $\lambda_{2}(t_{f})=1$ and
from the adjoint equation (\ref{adjoint1-adjoint2}) one deduces that
$\lambda_{2}(t)>0$ at any time. The switching function is such that
\[
\dot \phi=\lambda_{2}\mu^{\prime}(S_{2})(M-S_{1}-S_{2})+\min(0,f^{\prime}(S_{1})\phi),
\]
with $\phi(t_{f})=1$, from the transversality condition (\ref{transversal}).

Consider an optimal trajectory that reaches the target at time
$t_{f}>0$. The following properties are fulfilled for any
time $t<t_{f}$.
\begin{itemize}
\item[-] if $\phi(t)< 0$ then  $u=1$ is optimal on $[0,t_{c})$ up
  to a commutation time $t_{c}<t_{f}$
  such that $\phi(t_{c})=0$ (otherwise the target cannot be be reached with
  $\phi(t_{f})=1$). 
\item[-] when $\phi(t)\geq 0$ with $S_{2}(t)<S^{\star}_{2}$, then $S_{2}(t')<S^{\star }$ and $\phi(t')>0$ for any
$t \in (t,t_{f})$ i.e. $u=0$ is optimal on $(t,t_{f})$.
\item[-] when $\phi(t)\geq 0$ with $S_{2}(t)>S^{\star}_{2}$, $S_{2}$ and
  $\phi$ are decreasing up to a time $t'$ such that
  $S_{2}(t')=S^{\star}_{2}$ and $u=0$ is optimal on $(t,t')$.
\end{itemize}
The existence of a commutation time $t_{c}$ can be determined by the
backward integration of equations \eqref{dynS1-dynS2}-(\ref{adjoint1-adjoint2}) with the
constant control $u=0$ up to a possible time for which the switching
function $\phi$ is equal to zero. Remind that the Hamiltonian 
(\ref{Hamiltonian}) has to be
identically equal to zero along the optimal trajectories. 
Then, for each $S_{1} \in [0,\ul S_{1})$, $H=0$ gives
$\lambda_{0}=\mu(\ul S_{2})(M-S_{1}-\ul S_{2})$ for the optimal
trajectories that reach the target at $(S_{1},\ul S_{2})$. As long as
$u=0$ is optimal, $S_{1}$ is constant and one can also write from $H=0$
with $\lambda_{0}>0$ (the absence of abnormal extremal is given by
Corollary \ref{corollary2}):
\[
\lambda_{2}=\frac{\lambda_{0}}{\mu(S_{2})(M-S_{1}-S_{2})}=
\frac{\mu(\ul S_{2})(M-S_{1}-\ul S_{2})}{\mu(S_{2})(M-S_{1}-S_{2})}.
\]
Then, $(S_{2},\phi)$ is solution of the Cauchy problem:
\[
\left\{
\begin{array}{llll}
\dot S_{2} & = & -\mu(S_{2})(M-S_{1}-S_{2}), & S_{2}(t_{f})=\ul
S_{2},\\[2mm]
\dot \phi & = & \ds \frac{\mu^{\prime}(S_{2})\mu(\ul S_{2})(M-S_{1}-\ul
  S_{2})}{\mu(S_{2})}, & \phi(t_{f})=1,
\end{array}
\right.
\]
where $S_{1}<\ul S_{1}$ is fixed.
As the solution $S_{2}(\cdot)$ of this dynamics is strictly increasing, one can
parameterize the solution $\phi(\cdot)$ by $S_{2}\in [\ul
S_{2},M-S_{1})$ instead of the time, which amounts to write that
$S_{2}\mapsto \phi$ is solution of the Cauchy problem
\begin{equation}
\label{dphidS2}
\frac{d\phi}{dS_{2}} =  -\frac{\mu^{\prime}(S_{2})\mu(\ul S_{2})(M-S_{1}-\ul
  S_{2})}{\mu(S_{2})^{2}(M-S_{1}-S_{2})},
\end{equation} 
with the boundary condition $\phi(\ul S_{2})=1$.
Then a necessary and sufficient condition for the existence of a
commutation time $t_{c}$ is the existence of $S_{2}^{c}\in (\ul S_{2},M-S_{1})$ such that $\phi(S_{2}^{c})=0$  or
equivalently to have $\varphi(S_{1},S_{2}^{c})=1$ (that is exactly the condition $(S_{1},S_{2}^{c}) \in {\cal
  C}_{0}$). Consequently, for any initial condition in
${\cal E}_{0}$ with $S_{1}^{0}<S_{1}^{\star}$, the only possibility to reach the
target with $\phi=1$ is to choose the constant control $u=0$:
\begin{itemize}
\item[-] when the initial condition is the interior of ${\cal E}_{0}$, $\phi(\cdot)$ is always positive.
\item[-] for initial conditions in ${\cal C}_{0}$ or such that
$S_{1}^{0}=S_{1}^{\star}$ (when $S_{1}^{\star}>0$), $\phi(\cdot)$ is positive excepted at one isolated time for which
it is null (when the state leaves ${\cal C}_{0}$ or passes through the
end singular state).
\end{itemize}
We also deduce that $\phi=0$ at any state in ${\cal C}_{0}$ with $S_{1}<\ul S_{1}$.\\

Finally we consider initial conditions in ${\cal E}_{0}$ with
$S_{1}^{0}=\ul S_{1}$. 
As any trajectory from such initial condition is such that $S_{1}(t)\leq
\ul S_{1}$, any optimal trajectory is also optimal for the problem
with an augmented target
${\cal T}^{\prime}:=\{(S_{1},S_{2})\in [0,\ul S_{1}^{\prime}]\times [0,\ul
S_{2}] \}$ such that $\ul S_{1}^{\prime}>\ul S_{1}$. Then, the former
argumentation allows to conclude that the constant control $u=0$ is
also optimal for such initial condition, and thus one has $\phi(t)\geq
0$ at any time $t \in [0,t_{f}]$.
The transversality condition (\ref{transversal}) for such trajectories that reach
the target at the corner state $(\ul S_{1},\ul S_{2})$ gives
$\phi(t_{f})=1-2\alpha$ (where $\alpha \in [0,1]$). 
So, one deduces that $\phi(t_{f})$ is less than one and that the
backward integration of  equations
\ref{dynS1-dynS2}-(\ref{adjoint1-adjoint2}) with the
constant control $u=0$ (which amounts to solve the differential equation
(\ref{dphidS2}) with a boundary condition $\phi(\ul S_{2})$ in
$[0,1]$) gives the existence of a commutation time $t_{c}$ such that
$S_{2}(t_{c})\leq S_{2}^{c}(\ul S_{1})$.
We conclude that at state $(\ul S_{1},S_{2}^{c}(\ul S_{1}))$ one
should have also $\phi=0$.\qed

\section{Optimal synthesis in the admissible case}
\label{sec-admissible}

We first give a global characterization of the optimal solutions
when $S^{\star}_{2}\geq M$.

\begin{proposition}
\label{propositionMonod}
Assume that Hypotheses {\bf{H0}} and {\bf{H1}} are fulfilled with $S^{\star}_{2}\geq M$.
Then the feedback
\[
u^{\star}[S_{1},S_{2}]=\left|\begin{array}{ll}
    0 & \mbox{when } (S_{1},S_{2})\in {\cal E}_{0}\\
    1 & \mbox{otherwise}
  \end{array}\right.
\]
is optimal.
\end{proposition}

\proof
When the state is in ${\cal E}_{0}$ or in ${\cal
  Z}_{1}$, Propositions \ref{proposition1} and \ref{proposition2} give already the announced result.

If the constant control $u=1$ is not optimal outside these two
subsets, there should exists a time $t$ such that $\phi(t)\geq 0$
before reaching the target. According to Corollary \ref{corollary3},
  the optimal trajectory reaches ${\cal T}$ or ${\cal Z}_{1}$ with the
  constant control $u=0$ and $\phi>0$. This implies that
\begin{itemize}
\item[-] either the trajectory crosses ${\cal C}_{0}$ which is,
  according to Proposition \ref{proposition2}, a locus for which $\phi=0$, 
  thus a contradiction,
\item[-] either the trajectory reaches ${\cal Z}_{1}$ with $\phi>0$
  and then switching to $u=1$ at the boundary of ${\cal Z}_{1}$ is not possible, thus
  a contradiction with Proposition \ref{proposition1}.\qed
\end{itemize}

\bigskip

We give now a characterization of the optimal solutions of the problem
when $S^{\star}_{2}< M$, but under an assumption of the admissibility
of the singular arc up to the end singular state $S_1^\star$ defined in
(\ref{end_singular}).
We recall that $S_{1}^{\min}$ is the root of the function
$\nu(\cdot)$ defined in (\ref{nu}).

\begin{proposition}
\label{propositionHaldane1}
  Assume that Hypotheses {\bf{H0}} and {\bf{H1}} are fulfilled with $S^{\star}_{2}< M$.
  When $S_{1}^{\min}\leq S_{1}^{\star}$, the feedback
  \[
  u^{\star}[S_{1},S_{2}]=\left|\begin{array}{ll}
1  & \mbox{when } (S_{1},S_{2})\notin{\cal E}_{0} \mbox{ and }
S_{2}<S^{\star}_{2} \mbox{ or } (S_{1},S_{2})\in{\cal Z}_{1}\\
u_s(S_{1}) & \mbox{when } S_{1}>S_{1}^{\star} \mbox{ and } S_{2}=S^{\star}_{2}\\
0 & \mbox{otherwise}
\end{array}\right.
\]
is optimal, where $u_s(\cdot)$ is given in (\ref{singular_feedback}).
\end{proposition}

\proof
When the state is in ${\cal E}_{0}$ or in ${\cal
  Z}_{1}$, Propositions \ref{proposition1} and \ref{proposition2} give
already the optimality of the feedback. We consider now initial states
$S^{0}$ outside these two sets. Let $t_{f}$ be the minimal time to
reach the target.

If $S_{2}^{0}>S_{2}^{\star}$, let us show that the constant control
$u=0$ is optimal until the state reaches $S_{2}=S_{2}^{\star}$ or
${\cal Z}_{1}$ (notice that from outside ${\cal E}_{0}$, it is not
possible to reach ${\cal E}_{0}$ with the constant control $u=0$, or
$S_{2}=S_{2}^{\star}$ is reached before ${\cal E}_{0}$ due to property iii. of Lemma \ref{Lemma2}).
If not, there should exist a time $t<t_{f}$ with $\phi(t)\leq 0$ and
$S_{2}(t)>S_{2}^{\star}$. According to Corollary \ref{corollary3},
the optimal trajectory reaches $S_{2}=S_{2}^{\star}$ or ${\cal
  E}_{0}$ with the constant control $u=1$ and a negative value of $\phi$.
Note that the hypothesis $S_{1}^{\min}\leq S_{1}^{\star}$ implies that
for any  $S_{1}\in (S_{1}^{\star},M-S^{\star}_{2})$, one has at $S_{2}=S^{\star}_{2}$
with the control $u=1$
\[
\dot S_{2} =
f(S_{1})-\mu(S^{\star}_{2})(M-S_{1}-S^{\star}_{2})
\geq
f(S_{1}^{\star})-\mu(S^{\star}_{2})(M-\ul S_{1}^{\star}-S^{\star}_{2})\geq
0.
\]
Consequently the subset $\mathcal{S}\subset \Delta$ defined by:
\[
{\cal S}:=\{ (S_{1},S^{\star}_{2}) \, \vert \,
S_{1}\in [S_{1}^{\star},M-S_{2}^{\star}]\},
\]
is not reachable from above with the control $u=1$. 
On ${\cal E}_{0}$, $\phi$ has to be non-negative,
as $u=0$ is optimal by Proposition \ref{proposition2}, thus a contradiction.

Consider now an initial condition with $S^{0}_{2} < S^{\star}_{2}$ and
let us show that $u=1$ is optimal until reaching
$S_{2}=S^{\star}_{2}$ or the set ${\cal E}_{0}$ (notice that it is not
  possible to reach ${\cal Z}_{1}$ with the control $u=1$). 
  If not, there should exist a time $t<t_{f}$ with $\phi(t)\geq 0$ and
  $S_{2}(t)<S_{2}^{\star}$. According to Corollary \ref{corollary3},
  the optimal trajectory reaches ${\cal T}$ or ${\cal Z}_{1}$ with the
  constant control $u=0$ and $\phi>0$. This implies that
\begin{itemize}
\item[-] either the trajectory crosses ${\cal C}_{0}$ (when it is not empty) with
  $\phi>0$, thus a contradiction with Proposition \ref{proposition2},
\item[-] either the trajectory reaches ${\cal Z}_{1}$ with $\phi>0$
  and thus again a contradiction with the optimality of $u=1$
  in ${\cal Z}_{1}$ given by Proposition \ref{proposition1}.
\end{itemize}

Finally, we consider initial condition with $S^{0}_{2} =
S^{\star}_{2}$. If the trajectory leaves $S_{2}=S_{2}^{\star}$ before
reaching ${\cal E}_{0}$ or in ${\cal Z}_{1}$, there should exist a
time $t$ with one of the properties:
$S_{2}(t)<S_{2}^{\star}$ and $u(t)=0$, or $S_{2}(t)>S_{2}^{\star}$ and
$u(t)=1$. Corollary \ref{corollary3} implies then one of these
properties have also to be fulfilled on a interval $[t,t^{\prime})$
contradicting the above optimality obtained on both side of $S_{2}=S_{2}^{\star}$.
\qed
\smallskip

Optimal trajectories associated to the feedback control law provided by Proposition \ref{propositionHaldane1} 
are depicted on Fig. \ref{fig4a} and \ref{fig4b} in Section
\ref{numeric-sec}.



\section{Optimal Synthesis in the non admissible case}
\label{sec-opti2}
Our aim in this section is to study the optimal synthesis 
in the cases that are not covered by Propositions \ref{propositionMonod}
and \ref{propositionHaldane1}, that is when $S_{2}^{\star}\in
(\underline S_{2},M)$ and $S_{1}^{\min}>S_{1}^{\star}$ (recall from
point {\em iv.} of Lemma \ref{Lemma2} that when $S_{2}^{\star}\leq \underline S_{2}$ 
one has necessarily $S_{1}^{\min}\leq S_{1}^{\star}$ and thus this case
is already covered by Proposition \ref{propositionHaldane1}). Throughout this section, we suppose that Hypotheses {\bf{H0}} and 
{\bf{H1}} are satisfied.

As we already know the optimal synthesis in the {\it{extended target
    set }} $\mathcal{E}_0$ and in $\mathcal{Z}_1$, we only have to determine the optimal
feedback control in the set 
$(\mathcal{D} \backslash
\mathcal{T})\setminus(\mathcal{E}_0\cup\mathcal{Z}_{1})$.
Notice that the switching function should vanish 
on the boundary of $\mathcal{E}_0$.

Recall now that the singular feedback 
control satisfies $u_s(M-S_2^\star)= 0$ and that the mapping $S_1 \longmapsto u_s(S_1)$ is decreasing over $(0,M-S_2^\star]$ with $u_s(0^+)=+\infty$. Moreover, $S_1^{\min}$ 
corresponds to the unique point such that $u_s(S_1^{\min})=1$. 
So, we now suppose that:
\begin{equation}{\label{casDur}}
S_1^{\min}> S_1^\star.
\end{equation}
Hence, we have $u_s(S_1)>1$ for any value of $S_1$ such that $S_1<S_1^{\min}$ and 
the singular arc is admissible (i.e. $0 \leq u_s\leq 1$) only over the interval $[S_1^{\min},M-S_2^\star]$.
When $S_1<S_1^{\min}$ and $S_2=S_2^\star$, we have a saturating phenomena and singular trajectories no longer exist. The 
part of the singular arc where the inequality $u_s(S_1)> 1$ holds is usually called {\it{barrier}}. 
Hence, singular trajectories cannot reach the {\it{extended target}} $\mathcal{E}_0$, and this will affect the optimal synthesis 
(unlike when $S_1^{\min}\leq S_1^\star$ where singular trajectories can reach $\mathcal{E}_0$,  see Proposition \ref{propositionHaldane1}).

This situation has been encountered in different settings (see e.g. \cite{BayenDCDS,schaettler,schaettler2,ledje2,bonnard3} and references herein).
From a practical point of view, this means that the maximal admissible control does not guarantee a trajectory to stay on the singular arc. 

We will see that a singular trajectory starting at some point $(S_1,S_2^{\star})$ with $S_1 \in (S_1^{\min},M-S_2^\star)$ 
leaves the singular arc with the maximal control $u=1$
before reaching the point $(S_1^{\min},S_2^\star)$. This phenomenon is known as {\it{prior saturation}} \cite{bonnard3,schaettler}.
This means that for initial conditions $(S_1,S_2)$ 
on the singular arc (such that $M-S_1-S_2$ is sufficiently small), optimal trajectories are singular only until the point of prior 
saturation. The optimal control then switches to $u=1$ until reaching the extended target. 

For sake of completeness, we  provide a proof of this result adapted to our context and that will be useful to provide the optimal synthesis. 
Recall that a singular trajectory satisfies $S_2(t)=S_2^\star$ and $\dot{S}_1<0$ (provided that it is admissible, i.e. $S_1\in [S_1^{\min},M-S_2^\star)$).

\subsection{Optimality result when $S_1^{\star}>0$}
For technical reasons we suppose in addition that $S_1^{\star}$ satisfies: 
\begin{equation}{\label{hypDifficult}}
S_1^\star>0.
\end{equation}
The optimal synthesis when this assumption is not satisfied is discussed in Section \ref{pl}.
\begin{proposition}
There exists a unique point $\overline{S}_1\in (S_1^{\min},M-S_2^\star)$ such that any singular trajectory defined over the set $[S_1,S'_1]\times \{S_2^\star\}$ with 
$S_1^{\min}\leq S_1<S'_1\leq \overline{S}_1$ is not optimal. 
\end{proposition}

\begin{proof}
Consider a singular trajectory starting at an initial point $S_1^0>S_1^{\min}$. If the trajectory reaches the point $(S_1^{\min},S_2^\star)$ at time $t_0$, then it  satisfies $\phi(t_0)=\dot{\phi}(t_0)=0$. Now, there exists $\eps>0$ small enough such that for $t\in (t_0,t_0+\eps]$, the trajectory satisfies $S_2(t)<S_2^\star$. Using that $\dot{\phi}=\lambda_2\mu'(S_2)(M-S_1-S_2)$, one obtains that $\dot{\phi}>0$ in $(t_0,t_0+\eps]$ (recall that $\lambda_2>0$), and therefore we have $u=0$ in $[t_0,t_0+\eps]$. Hence, the trajectory cannot switch to $u=1$ at any time $t\geq t_0+\eps$ as we would have $\dot{\phi}(t)\leq 0$  in contradiction with $\dot{\phi}(t)=\frac{\mu'(S_2(t))}{\mu(S_2(t))}>0$ and $\phi(t)=0$.
It follows that for any $t\geq t_0+\eps$, we have $u=0$. 
Then, either the trajectory does not reach the target (if $S_1^{\min}>\underline{S}_1$) or the trajectory 
cannot satisfy $\phi=0$ on $\mathcal{C}_0$ (if $S_1^\star<S_1^{\min}\leq \underline{S}_1$). 
We have thus proved that a singular trajectory connecting a point $(S_1,S_2^\star)$ with $S_1 \in (S_1^{\min},M-S_2^{\star})$ 
to the point $(S_1^{\min},S_2^\star)$ is not optimal.

By a similar reasoning, we obtain that an optimal trajectory which is singular over a time interval 
will not leave the singular arc at some point $S_1>S_1^{\min}$ with $u=0$.
Hence, there exists a point $\overline{S}_1\in (S_1^{\min},M-S_2^{\star}]$ such that any singular extremal trajectory starting from 
 $(S_1,S_2^\star)$ with $S_1\in (\overline{S}_1,M-S_2^{\star})$ will switch to $u=1$ at the point $(\overline{S}_1,S_2^\star)$.

Now using \eqref{hypDifficult}, one can consider the solution of \eqref{dynS1-dynS2} backward in time with $u=1$ from $(S_1^\star,S_2^\star)$. 
As we have $\dot{S}_1>0$, we can parameterize this curve $(S_1(\cdot),S_2(\cdot))$ 
as the graph of a $C^1$-mapping $S_1\longmapsto \xi^\star(S_1)$ defined for $S_1\geq S_1^\star$.
Hence, $\xi^\star$ is the unique solution of the Cauchy problem:
\begin{equation}{\label{curve-xi}}
\frac{d s_2}{d s_1}=-1+\frac{\mu(s_2)(M-s_1-s_2)}{f(s_1)}, \; \; s_2(S_1^\star)=S_2^\star.
\end{equation}
From the definition of $S_1^{\min}$, this trajectory cannot intersect the singular arc at some point $S_1\in (S_1^\star,S_1^{\min})$. Moreover, the trajectory cannot leave $\mathcal{D}$ through the set $\{(S_1,S_2) \in \R_+\times \R_+ \; ; \; S_1+S_2=M\}$ that is invariant by \eqref{dynS1-dynS2}. Hence, it will cross the singular arc at some point $\tilde{S}_1\in (S_1^{\min},M-S_2^\star)$. 

Fix a point $\hat S_1\in (\tilde S_1,M-S_2^{\star})$. For $S_1 \in  [S_1^{\min},\hat S_1]$, let us denote by $\gamma_{S_1}$ a singular extremal 
trajectory connecting $(\hat S_1,S_2^\star)$ to the point $(S_1,S_2^\star)$, and define 
a set $\mathcal{S}$ by:
$$
\mathcal{F}:=\{S_1 \in [S_1^{\min},M-S_2^\star]\; ; \; \gamma_{S_1} \; \mathrm{is} \; \mathrm{optimal} \; \mathrm{over} \; [S_1,\hat S_1]\}.
$$
Let us show that $\mathcal{F}$ is non-empty. Consider a singular trajectory starting at $(S_1,S_2^{\star})$ with $\hat S_1\leq S_1<M-S_2^\star$. 
We know that it is not optimal for the singular trajectory to leave the singular arc with $u=0$. 
Moreover, if the singular trajectory leaves the singular arc with $u=1$ before reaching $\tilde{S}_1$, then the trajectory cannot switch on the extended target $\mathcal{E}_0$, and we have a contradiction. In fact, such a trajectory necessarily reaches $\mathcal{E}_0$ at a point $S_2>S_2^\star$ by definition of $\xi^\star$. At this point, the switching function is such that $\phi<0$ which is not possible ($\phi$ has to be zero on the boundary of $\mathcal{E}_0$). 

We have thus proved that for any $S_1$ such that $\hat S_1\leq S_1<M-S_2^\star$, a singular trajectory starting at $(S_1,S_2^{\star})$ is optimal at least until reaching the point 
$(S_1^{\min},S_2^{\star})$, so $\mathcal{S}\not= \emptyset$. 
Now, the set $\mathcal{F}$ is clearly an interval and we take for $\overline{S}_1$ the infimum of $\mathcal{F}$. This proves the result as we know that 
$\overline{S}_1>S_1^{\min}$.
\qed
\end{proof}

Notice that \eqref{hypDifficult} is crucial for defining the point $(\overline{S}_1,S_2^\star)$ of 
{\it{prior saturation}}. 
The next proposition characterizes the number of switching times for trajectories starting above the singular arc with $S_1\in (S_1^\star,S_1^{\min})$ and it 
will allow us to define the switching curve emanating from $(\overline{S}_1,S_2^\star)$.
\begin{proposition}{\label{prop1-synthesis}}
Consider a point $S^0=(S_1^0,S_2^0)$ such that $S_1^0\in (S_1^\star,\overline{S}_1)$ and $S_2^0> \xi^\star(S_1^0)$. 
Then, any optimal trajectory $\gamma$ 
steering $S^0$ to the extended target ${\cal E}_{0}$ has a unique switching time $t_0$ such that $S_2(t_0)> S_2^\star$ and we have $u(t)=0$ for $t\in [0,t_0]$ and 
$u(t)=1$  for $t>t_0$.
\end{proposition}

\begin{proof}
First, recall that an extremal trajectory cannot switch from $u=1$ to $u=0$ at a  time $t_0$ such that 
$S_2(t_0)>S_2^\star$. Hence, the number of switching times of $\gamma$ before reaching the singular arc is either $0$ or $1$. 
Now, take $S_1^0\in (S_1^\star,\overline{S}_1)$. By using a similar reasoning as in the previous proof, we know that if we have $u=0$ until reaching the singular arc, then $\gamma$ is not optimal. Notice that $\overline{S}_1\leq \tilde S_1$ as the trajectory starting from $(\tilde S_1,S_2^{\star})$ with $u=1$ until $\mathcal{E}_0$ is not optimal.
It follows that if the control switches to $u=1$ at a point $(S_1,S_2)$ with $S_2>\xi^\star(S_1)$, 
then the trajectory will reach the set $\mathcal{E}_0$ at a point $S_2>S_2^\star$ (using that $S_2^0>\xi^\star(S_1^0)$), 
and we have a contradiction with $\phi=0$ at the boundary of $\mathcal{E}_0$. 
Hence, there exists a unique switching time $t_0$ from $u=0$ to $u=1$ and  $S_2(t_0)> S_2^\star$.
\qed
\end{proof}

For $S_1\in (S_1^\star,\overline{S}_1)$, we denote by $S_2:=\zeta(S_1)\geq S_2^\star$ the unique switching point from $u=0$ to $u=1$ 
from an optimal trajectory starting at some point $(S_1,\xi^\star(S_1))$, 
and let $\mathcal{C}_1$ be the {\it{switching curve}} defined by:
$$
\mathcal{C}_1=\{(S_1,\zeta(S_1)) \; ; \; S_1 \in (S_1^{\star},\overline{S}_1)\}.
$$
From the classification of {\it{frame points}} and {\it{frame curves}} \cite{bosc}, 
the point of prior saturation $(\overline{S}_1,S_2^\star)$ is a frame point of type $(CS)_2$ 
at the intersection between 
the singular set and a switching curve. In fact, singular trajectories stop to be optimal
at this point and leave the singular set with the maximal control $u=1$. 
Therefore, 
we can extend $\zeta$ at the point $\overline{S}_1$ setting $\zeta(\overline{S}_1)=S_2^\star$.
\begin{remark}
(i) We can show by the arguments above that the switching curve passes through the point $(S_1^\star,S_2^\star)$ i.e. $\lim_{S_1 \rightarrow S_1^\star} \zeta(S_1)=S_2^\star$.
\\
(ii) We believe that the $\mathcal{C}_1$ is continuous. Unfortunately, this question seems 
difficult to address as we cannot easily obtain an implicit equation for $\mathcal{C}_1$ (such as for the $\mathcal{C}_0$). 
The difficulty comes from the fact that the initial system with $u=1$ leads to a non-autonomous differential equation for $S_2$ as a function of $S_1$.
Nevertheless, this property is not crucial in order to obtain the optimal synthesis.  
\end{remark}

\begin{proposition}{\label{prop2-synthesis}}
For any $S_1^0\in [\overline{S}_1,M-S_2^\star)$ and $S_2^0\geq S_2^\star$, any optimal trajectory starting at $(S_1^0,S_2^0)$ satisfies 
$u=0$ until reaching the singular arc. 
\end{proposition}
\begin{proof}
Suppose by contradiction that an optimal trajectory starting at some point $(S_1^0,S_2^0)$ with $S_1^0\in [\overline{S}_1,M-S_2^\star]$ and $S_2^0\geq S_2^\star$ satisfies $u=1$ over a time interval $[0,\tau]$, for some $\tau>0$. Recall that $\zeta(S_1^{\min})=S_2^\star$.
As $\overline{S}_1>S_1^{\min}$, there exists $\tau'>0$ such that we have 
$u=1$ over the time interval $[0,\tau']$ and such that $S_2(\tau')>\zeta(S_1(\tau'))$ (in particular, this trajectory cannot intersect the singular arc at some point $(S_1,S_2^\star)$ with $S_1\geq \overline{S}_1$). 
Thus we obtain a contradiction with Proposition \ref{prop1-synthesis}. 
In fact, we know that for $S\in \mathcal{D}$ such that $S_1<\overline{S}_1$ and $S_2>\zeta(S_1)$, one has necessarily $u=0$.
\qed
\end{proof}

It remains to study the case where initial conditions are taken below the singular arc.
\begin{proposition}{\label{prop3-synthesis}}
For any initial conditions $S^0=(S_1^0,S_2^0)\in \mathcal{D}\backslash \mathcal{E}_0$ and such that $S_2^0 <S_2^\star$, we have $u=1$.
\end{proposition}
\begin{proof}
We know that any optimal trajectory cannot switch from $u=0$ to $u=1$ at a point 
$S^0=(S_1^0,S_2^0)\in \mathcal{D}\backslash \mathcal{E}_0$ and such that $S_2^0 <S_2^\star$. 
Suppose now that an optimal trajectory starting at 
some point $S^0=(S_1^0,S_2^0)\in \mathcal{D}\backslash \mathcal{E}_0$ with $S_2^0 <S_2^\star$
satisfies $u=0$ over a time interval $[0,\tau]$. 
If $S_0>\underline{S}_1$, then the trajectory does not reach the target and we have a contradiction. 
Finally, if $S_0\leq \underline{S}_1$ (this case can be empty if $\mathcal{C}_0$ does not exist), then one should have $u=0$ until reaching $\mathcal{C}_0$, and we would have a contradiction with the fact that $\phi=0$ on $\mathcal{C}_0$.
\qed
\end{proof}

The next theorem summarizes the results of Propositions \ref{prop1-synthesis}, \ref{prop2-synthesis} and \ref{prop3-synthesis} and provides an optimal feedback control of the problem whenever \eqref{casDur}-\eqref{hypDifficult} are satisfied.

\begin{theo}{\label{main1}}
Assume that Hypotheses H0 and H1 are fulfilled. In addition, suppose that \eqref{casDur}-\eqref{hypDifficult} are satisfied. Then, an optimal feedback control steering the system in minimal time to the target is given by 
\begin{equation}{\label{feedback1}}
u^{\star}[S_1,S_2]:=
\left|
\begin{array}{ll}
0 & \; \; \mathrm{if} \; \; (S_1,S_2)\in \mathcal{E}_0 \; \; \mathrm{or} \; \; S_2\geq \max(\zeta(S_1),S_2^\star),\\
u_s(S_1) & \; \; \mathrm{if} \; \;  S_1\in [\overline{S}_1,M-S_2^\star)\; \; \mathrm{and} \; \; S_2=S_2^\star,\\
1 & \; \; \mathrm{if} \; \;  (S_1,S_2)\notin \mathcal{E}_0 \; \; \mathrm{and} \; \; S_2<S_2^\star.
\end{array}
\right.
\end{equation}
\end{theo}
Optimal trajectories corresponding to the feedback $u^{\star}$ are depicted on Fig. \ref{fig1} when $\mathcal{C}_0=\emptyset$ and on Fig. \ref{fig2a}, \ref{fig2b}, and \ref{fig2c} when $\mathcal{C}_0\not=\emptyset$. 
\begin{remark}
The construction of $\mathcal{C}_1$ is explained in Section \ref{numeric-sec}. 
We observe numerically that the switching curve $\mathcal{C}_1$ can be non-smooth, see Fig. \ref{fig2c}. 
We believe that this is a consequence of the non-smoothness of the target set $\mathcal{E}_0$ at $(\underline{S}_1,\underline{S}_2)$.
\end{remark}
\subsection{Discussion when $S_1^\star=0$}{\label{pl}}
When \eqref{hypDifficult} is not satisfied, i.e.:
\begin{equation}{\label{hypDifficult2}}
S_1^\star=0,
\end{equation}
then, the existence of the switching curve $\mathcal{C}_1$ is not
straightforward using the previous arguments. Indeed, the construction
of this curve cannot be initiated from $(S_1^\star,S_2^\star)$ as we
do in the previous case with $S_1\longmapsto \xi^\star(S_1)$ (recall \eqref{curve-xi}), because $\{0\}\times [0,M]$ is invariant by the system.
Therefore, extremal trajectories starting above the singular arc with $u=1$ can be optimal until reaching the extended target ${\cal E}_{0}$. 
We obtain the following statement.
\begin{theo}{\label{main1bis}} Assume that Hypotheses H0 and H1 are fulfilled.
In addition, suppose that \eqref{casDur}-\eqref{hypDifficult2} are satisfied. Then, an optimal control $u$ steering the system in minimal time
from $(S_1^0,S_2^0)$ to the target satisfies the following:
\begin{enumerate}[label=\roman{*}.]
\item If the initial condition $(S_1^0,S_2^0)$ is such that $S_2^0> S_2^\star$, then there exists $t_0\geq 0$ such that $u=0$ on $[0,t_0]$, and then we have $u=1$ until reaching the set $\mathcal{E}_0$. 
\item If the initial condition $(S_1^0,S_2^0)$ is such that $S_2^0= S_2^\star$, then :
\begin{itemize}
\item[-] If $S_1^0 \leq S_1^{\min}$, then we have $u=1$ until reaching the set $\mathcal{E}_0$. 
\item[-] If $S_1^0>S_1^{\min}$, then either we have $u=1$ until reaching the set $\mathcal{E}_0$, or we have $u=u_s$ 
on some time interval $[0,t_0]$ with $t_0 \geq 0$, and then $u=1$ until reaching $\mathcal{E}_0$.
\end{itemize}
\item If $S_2^0<S_2^\star$, then an optimal control is given by Theorem \ref{main1}. 
\end{enumerate}
\end{theo}
\begin{proof}
The proof of i. and ii. is a consequence of Propositions \ref{prop1-synthesis} and \ref{prop2-synthesis} except that we cannot exclude trajectories with a constant control $u=1$ to be optimal until the set $\mathcal{E}_0$. Therefore, $t_0$ can be zero. 
The proof of iii. is the same as in Theorem \ref{main1}.
\qed
\end{proof}

Optimal trajectories are depicted on Fig. \ref{fig3a}, \ref{fig3b},
and \ref{fig3c} in Section \ref{numeric-sec}. 
We see numerically that there exists a switching curve $\mathcal{C}_1$ 
that satisfies similar properties as in the case $S_1^\star>0$:
\begin{itemize}
\item[-] The curve $\mathcal{C}_1$ is above the singular locus $\mathcal{S}$.
\item[-] The curve $\mathcal{C}_1$ connects the point of prior saturation to a point $(0, S'_2)$ with  $S'_2 \in (S_2^*,M)$.
\end{itemize}

\section{Numerical simulations and discussion}
\label{numeric-sec}
We have chosen for $f$ the linear functions $f(S_1):=S_1$,
and for the specific growth rate $\mu(\cdot)$, we have considered the
Haldane function:
\[
\mu(S_{2}):=\frac{\bar\mu S_{2}}{K_{s}+S_{2}+S_{2}^{2}/K_{i}}.
\]
One can straightforwardly check that Hypotheses {\bf{H0}} and {\bf{H1}} are satisfied 
with
\[
S_{2}^{\star}=\sqrt{K_{s}K_{i}} .
\]
One can notice that the Haldane function can be seen as a
generalization of the Monod function $\mu_m$ (which is monotonic and often
used in microbial growth) defined by:
\[
\mu_m(S_{2}):=\frac{\mu_{\max} S_{2}}{K_{s}+S_{2}},
\]
on a the interval $[0,M]$, taking large values of the parameter $K_{i}$.

We now explain how the curve $\mathcal{C}_1$ is computed numerically (Theorem \ref{main1} and Theorem \ref{main1bis}). 
The switching curve $\mathcal{C}_1$ is guaranteed by Theorem \ref{main1} (whenever $S_1^{\min}>S_1^\star>0$). Recall that the function $\phi$ vanishes both on $\mathcal{C}_1$ and on $\mathcal{C}_0$. In order to plot $\mathcal{C}_1$, we integrate backward in time the system with the maximal control $u=1$ from $\mathcal{C}_0$ ($\mathcal{C}_0$ is known explicitly). More precisely, the construction goes as follows.
Consider the dynamics:
\begin{equation}
\label{dyn_Qmax_S1}
\left\{
\begin{array}{lll}
\ds \frac{d\sigma_{2}}{d\sigma_{1}} & = & \ds -1+\frac{\mu(\sigma_{2})(M-\sigma_{1}-\sigma_{2})}{f(\sigma_{1})},\\[4mm]
\ds \frac{d\psi}{d\sigma_{1}} & = & \ds
-\frac{\mu^{\prime}(\sigma_{2})}{\mu(\sigma_{2})f(\sigma_{1})}-\psi\left(\frac{f^{\prime}(\sigma_{1})}{f(\sigma_{1})}+
\frac{\mu^{\prime}(\sigma_{2})}{\mu(\sigma_{2})}\right),
\end{array}
\right.
\end{equation}
with initial conditions
\begin{equation}
\label{BC_Qmax_S1}
(\sigma_{2}(\sigma_{10}),\psi(\sigma_{10}))=(\sigma_{20},0), \quad
(\sigma_{10},\sigma_{20}) \in {\cal S}_{0}.
\end{equation}
We shall denote $(\sigma_{2}^{(\sigma_{10},\sigma_{20})}(\cdot),
\psi^{(\sigma_{10},\sigma_{20})}(\cdot))$ its solutions. The previous system describes the evolution of 
$S_2$ and $\phi$ backward in time from $\mathcal{S}_0$. 
Now, define a mapping $\theta: \mathcal{S}_0 \rightarrow \mathbb{R}$ associating to any initial condition on $\mathcal{C}_0$ 
the value of $\sigma_1$ for which the solution of \eqref{dyn_Qmax_S1}-\eqref{BC_Qmax_S1} is such that
$$
\psi^{(\sigma_{10},\sigma_{20})}(\sigma_1))=0.
$$
From Theorem \ref{main1}, we know that there exists a non-empty subset $E\subset \mathcal{S}_0$ such that $\mathcal{C}_1$ is the 
image of $E$ by $\theta$. 
In order to compute numerically $\mathcal{C}_1$, we integrate the previous system and we stop the integration whenever $\psi$ vanishes, 
which corresponds to a switching point. If $\psi$ does not vanish, then we repeat this procedure 
by changing the initial condition on $\mathcal{S}_0$. The curve
$\mathcal{C}_1$ is depicted on Fig. \ref{fig1}
to \ref{fig3c}:


\begin{itemize}
\item Case I (see Fig. \ref{fig1}) corresponds to the case where \eqref{casDur} and \eqref{hypDifficult} are satisfied (optimal synthesis given by Theorem \ref{main1}). Moreover, in this case, $\mathcal{C}_0=\emptyset$, but $\mathcal{C}_1\not= \emptyset$.
\item Cases IIa, IIb and IIc (see Fig. \ref{fig2a}, \ref{fig2b}, and \ref{fig2c}) correspond to the case where \eqref{casDur} and \eqref{hypDifficult} are satisfied (optimal synthesis given by Theorem \ref{main1}). Moreover, we see in Fig. \ref{fig2a} and \ref{fig2b} that $S_1^{\min}$ and $S_1^\star$ can be less or greater than $\bar S_1$. Fig. \ref{fig2c} depicts a case where the switching curve $\mathcal{C}_1$ seems to be non-smooth (due to the non-smoothness of the target set at the corner point). 
\item Case IIIa, IIIb and IIIc (see Fig. \ref{fig3a}, \ref{fig3b}, and \ref{fig3c}) correspond to the case where \eqref{casDur} and \eqref{hypDifficult2} are satisfied (optimal synthesis given by Theorem \ref{main1bis}). Moreover, we see in Fig. \ref{fig3a} that $\mathcal{C}_1$ is defined only from points of $\mathcal{C}_0$ whereas in Fig. \ref{fig3b}, 
the curve $\mathcal{C}_1$ is defined both from points of $\mathcal{C}_0$ and from $\{\bar S_2\}\times [0, \bar S_1]$. In Fig. \ref{fig3c}, we observe that 
$\mathcal{C}_1$ is non-smooth (same property as in case IIc). 
\item Case IVa and IVb depict optimal trajectories as in Proposition \ref{propositionHaldane1} when $S_1^{\min}>\bar S_1$ and $S_1^{\min}<\bar S_1$.
\end{itemize}
Table \ref{table1} presents the values of the parameters for the different cases that have been simulated.
\begin{table}
\begin{center}
\begin{tabular}{|c|c|c|c|c|c|c|c|c|}
\hline
case & $\bar\mu$ & $K$ & $K_{i}$ & $1$ & $M$ & $\underline S_{1}$ & $\underline S_{2}$ &Optimal Synthesis\\
\hline
I & 1 & 2 & 0.23 & 0.1 & 1.3 & 0.15 & 0.05 & Theorem \ref{main1}\\
\hline
II a & 1 & 5 & 0.23 & 0.03 & 1.3 & 0.29 & 0.05 & Theorem \ref{main1}\\
\hline
II b & 1 & 3.5 & 0.23 & 0.04 & 1.3 & 0.14 & 0.02 & Theorem \ref{main1}\\
\hline
II c & 1 & 3.5 & 0.23 & 0.015 & 1.3 & 0.14 & 0.02 & Theorem \ref{main1} \\
\hline
III a & 30 & 4 & 0.7 & 5 & 2.4 & 0.2 & 0.02 & Theorem \ref{main1bis} \\
\hline
III b & 30 & 4 & 0.7 & 5 & 2.4 & 0.09 & 0.02 & Theorem \ref{main1bis} \\
\hline
III c & 30 & 4 & 0.7 & 5 & 2.4 & 0.05 & 0.02 & Theorem \ref{main1bis} \\
\hline
IV a & 1 & 2 & 0.23 & 0.1 & 1.3 & 0.15 & 0.8 & Proposition \ref{propositionHaldane1}\\
\hline
IV a & 1 & 2 & 0.23 & 1 & 1.3 & 0.15 & 0.8 & Proposition \ref{propositionHaldane1} \\
\hline
\end{tabular}
\caption{List of cases \label{table1}}
\end{center}
\end{table}

\begin{figure}
\begin{center}
\includegraphics[scale=0.35]{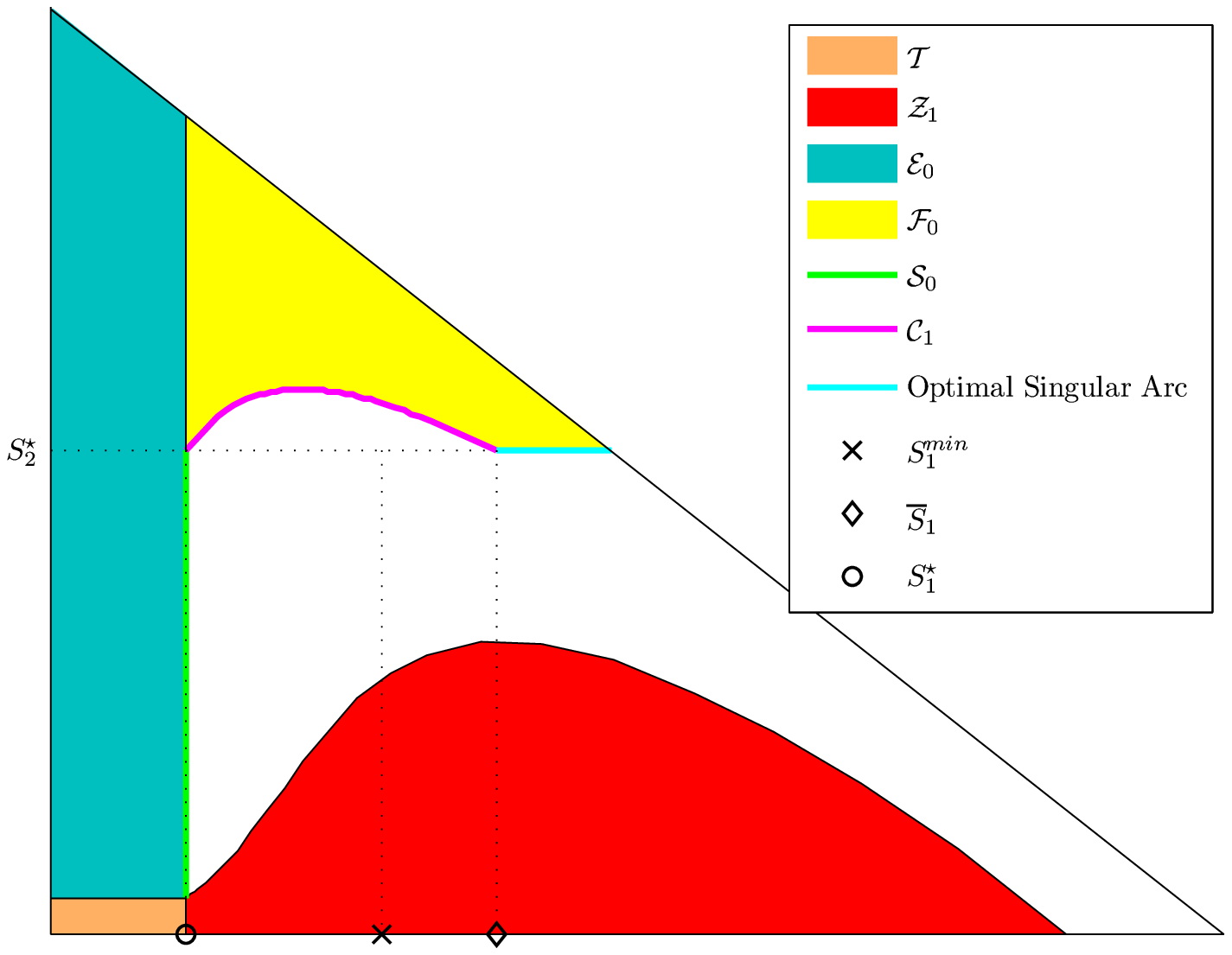} \hspace{4mm}
\includegraphics[scale=0.35]{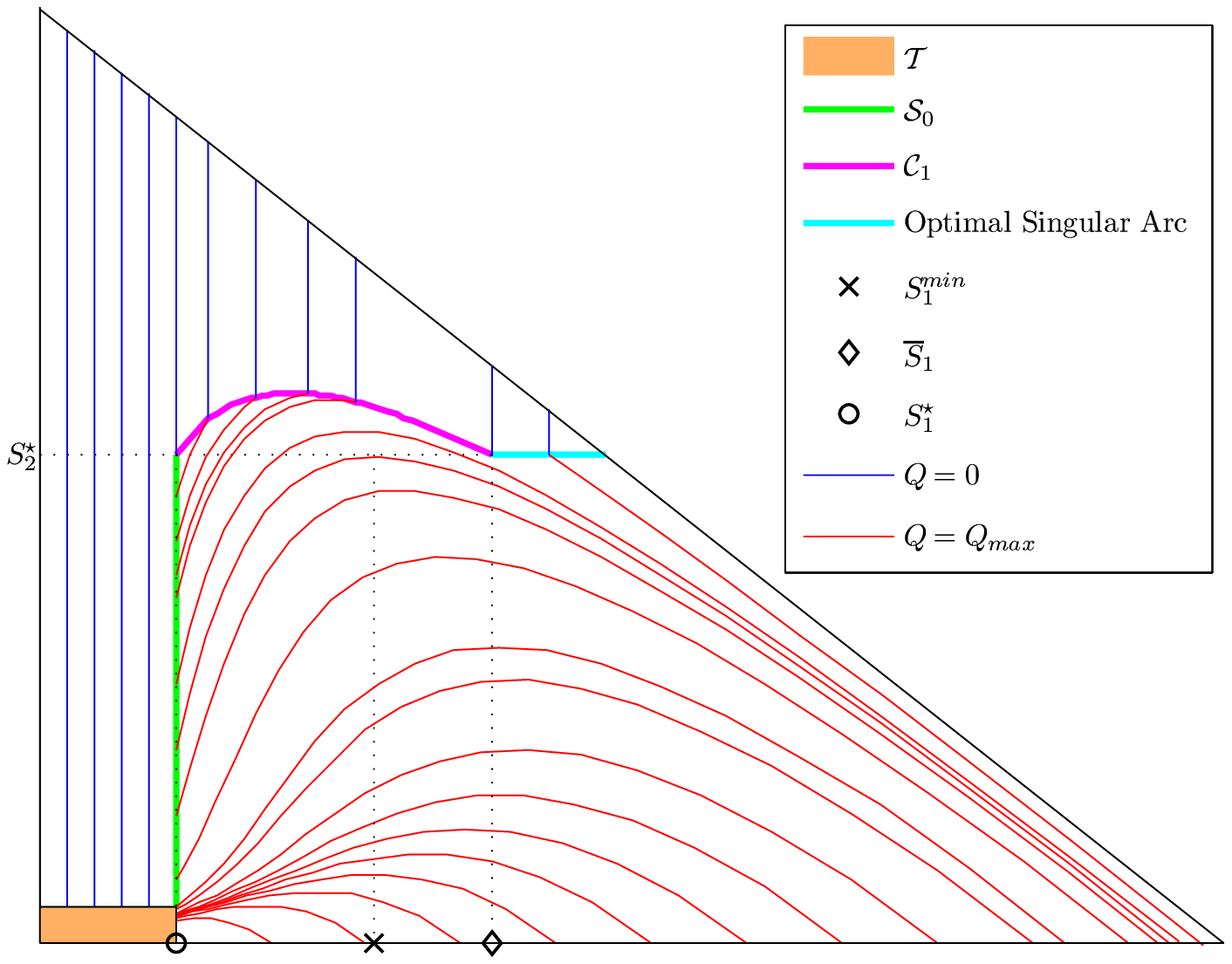}
\caption{Case I. {\it{Picture left}}: Partition of the state space. {\it{Picture right}}: Optimal synthesis provided by Theorem \ref{main1}).\label{fig1} }
\end{center}
\end{figure}

\begin{figure}
\begin{center}
\includegraphics[scale=0.35]{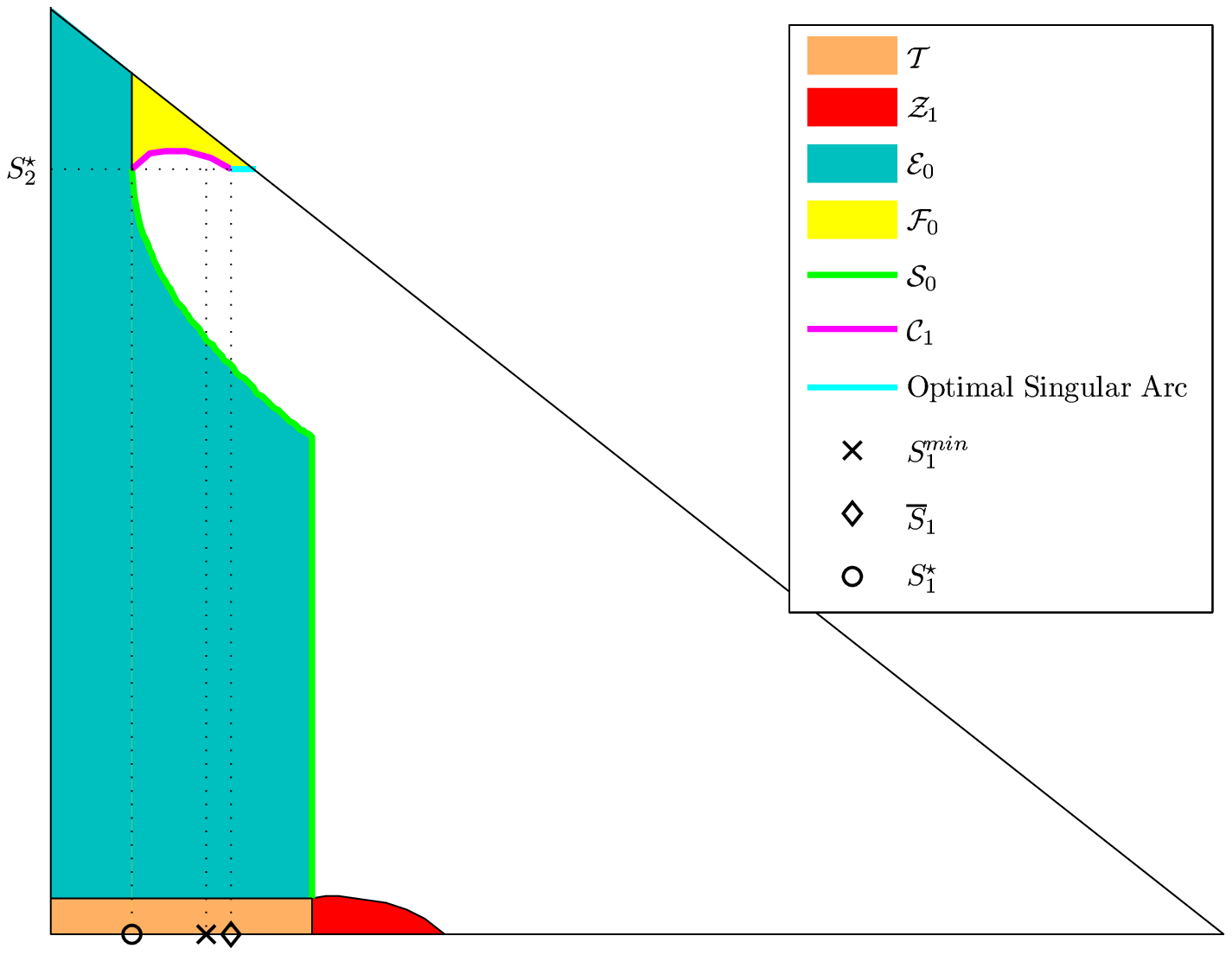} \hspace{4mm}
\includegraphics[scale=0.35]{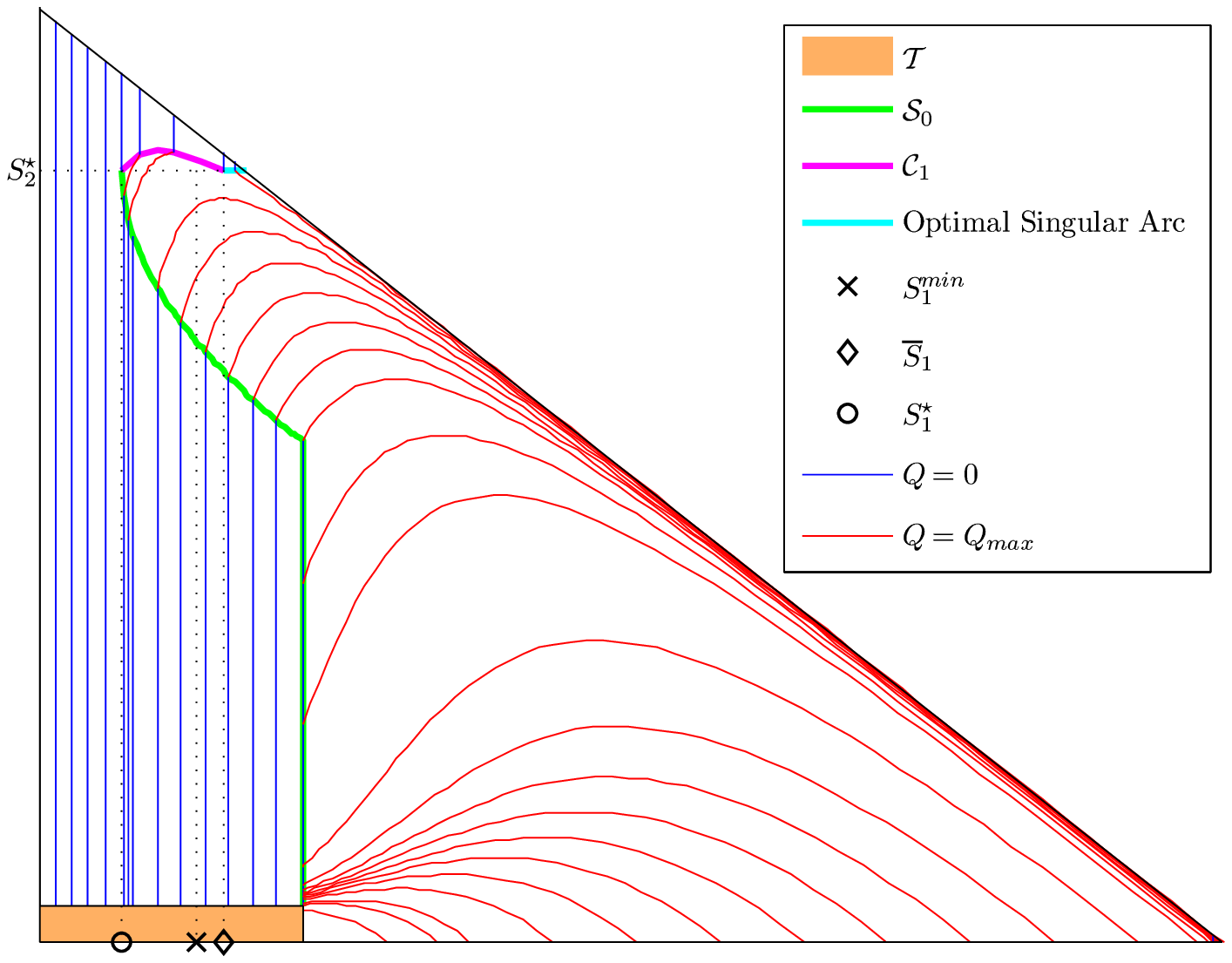}
\caption{Case IIa. {\it{Picture left}}: Partition of the state space. {\it{Picture right}}: Optimal synthesis provided by Theorem \ref{main1}). \label{fig2a}}
\end{center}
\end{figure}

\begin{figure}
\begin{center}
\includegraphics[scale=0.35]{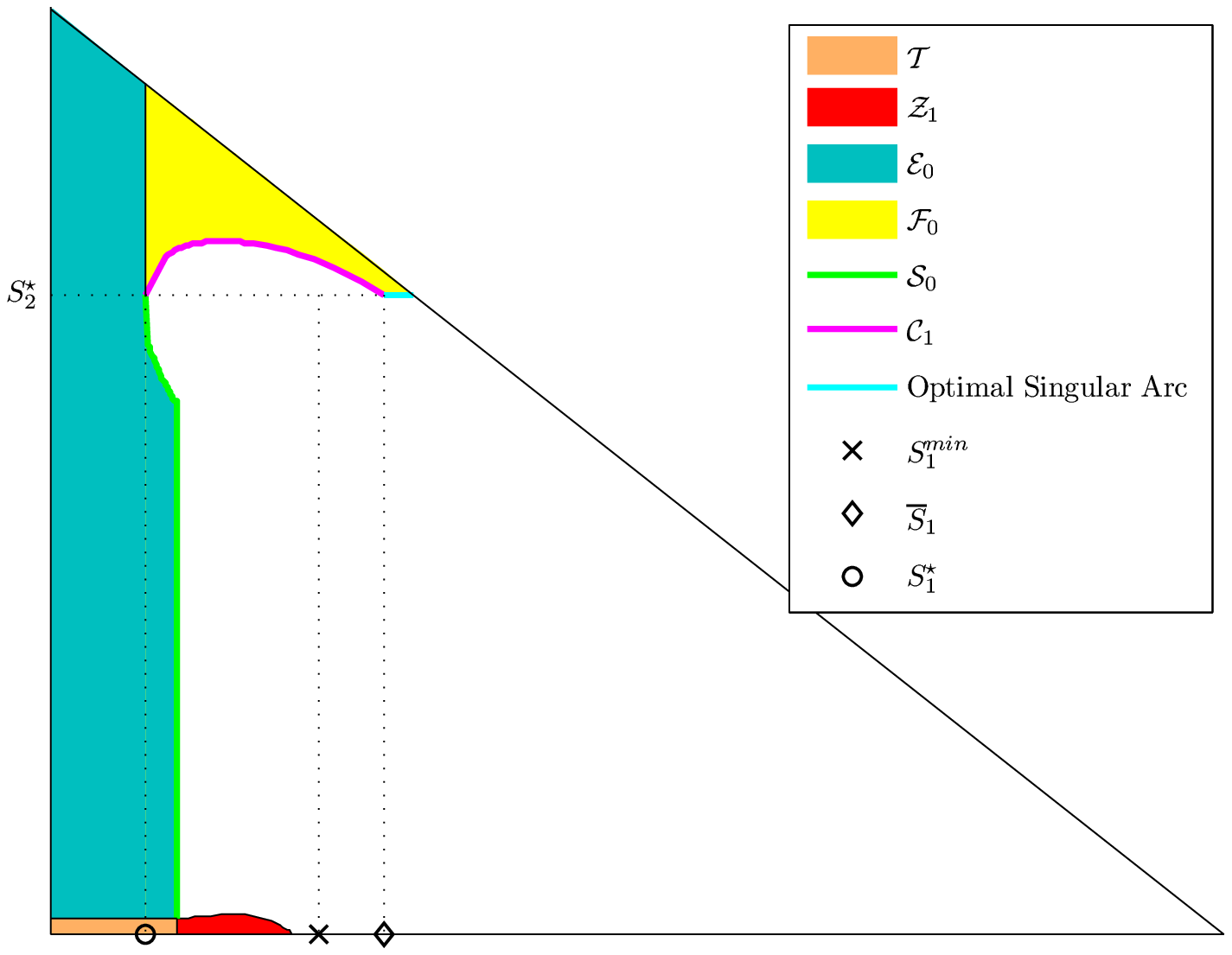} \hspace{4mm}
\includegraphics[scale=0.35]{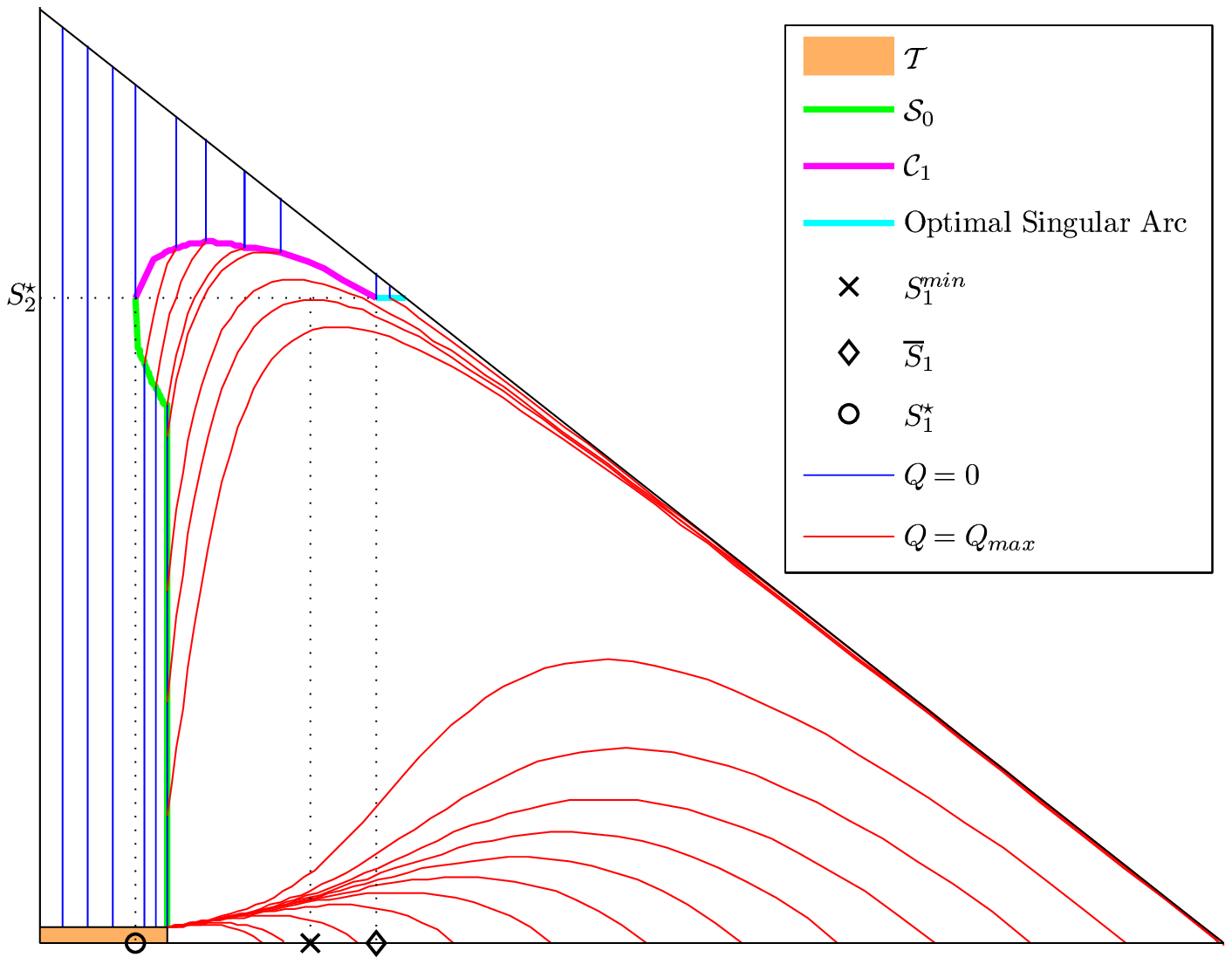}
\caption{Case IIb. {\it{Picture left}}: Partition of the state space. {\it{Picture right}}: Optimal synthesis provided by Theorem \ref{main1}). \label{fig2b}}
\end{center}
\end{figure}

\begin{figure}
\begin{center}
\includegraphics[scale=0.35]{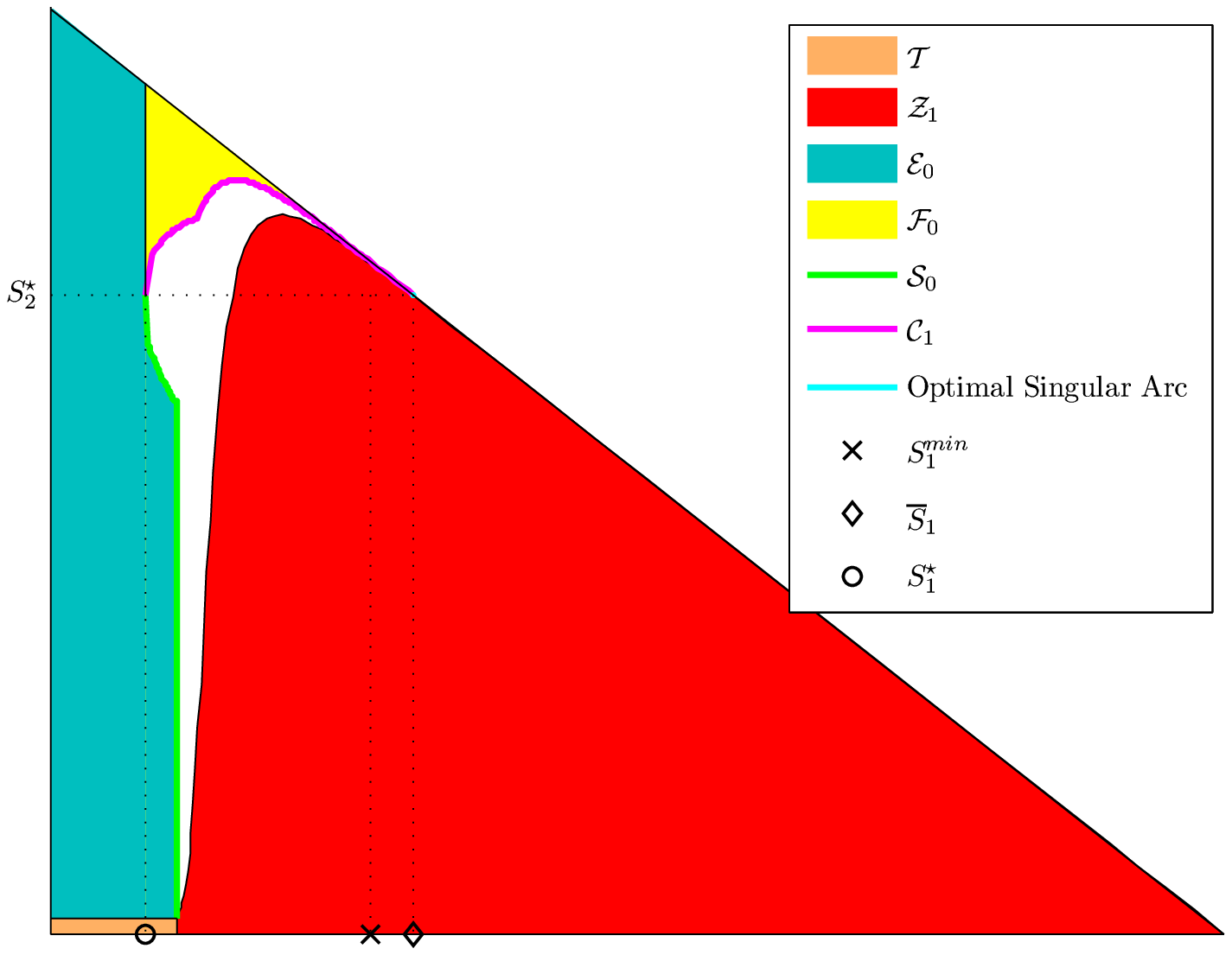} \hspace{4mm}
\includegraphics[scale=0.35]{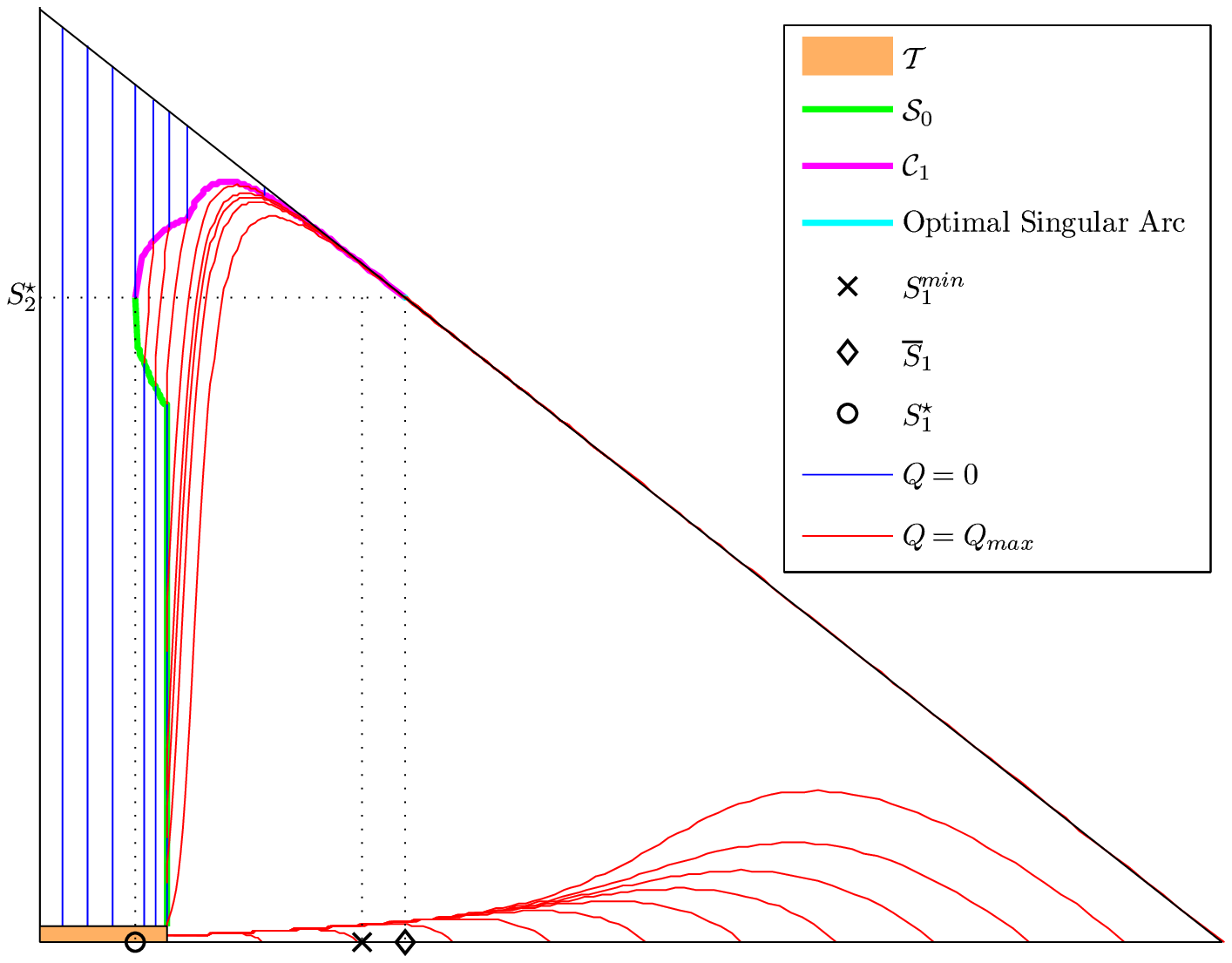}
\caption{Case IIc. {\it{Picture left}}: Partition of the state space. {\it{Picture right}}: Optimal synthesis provided by Theorem \ref{main1}). \label{fig2c}}
\end{center}
\end{figure}

\begin{figure}
\begin{center}
\includegraphics[scale=0.35]{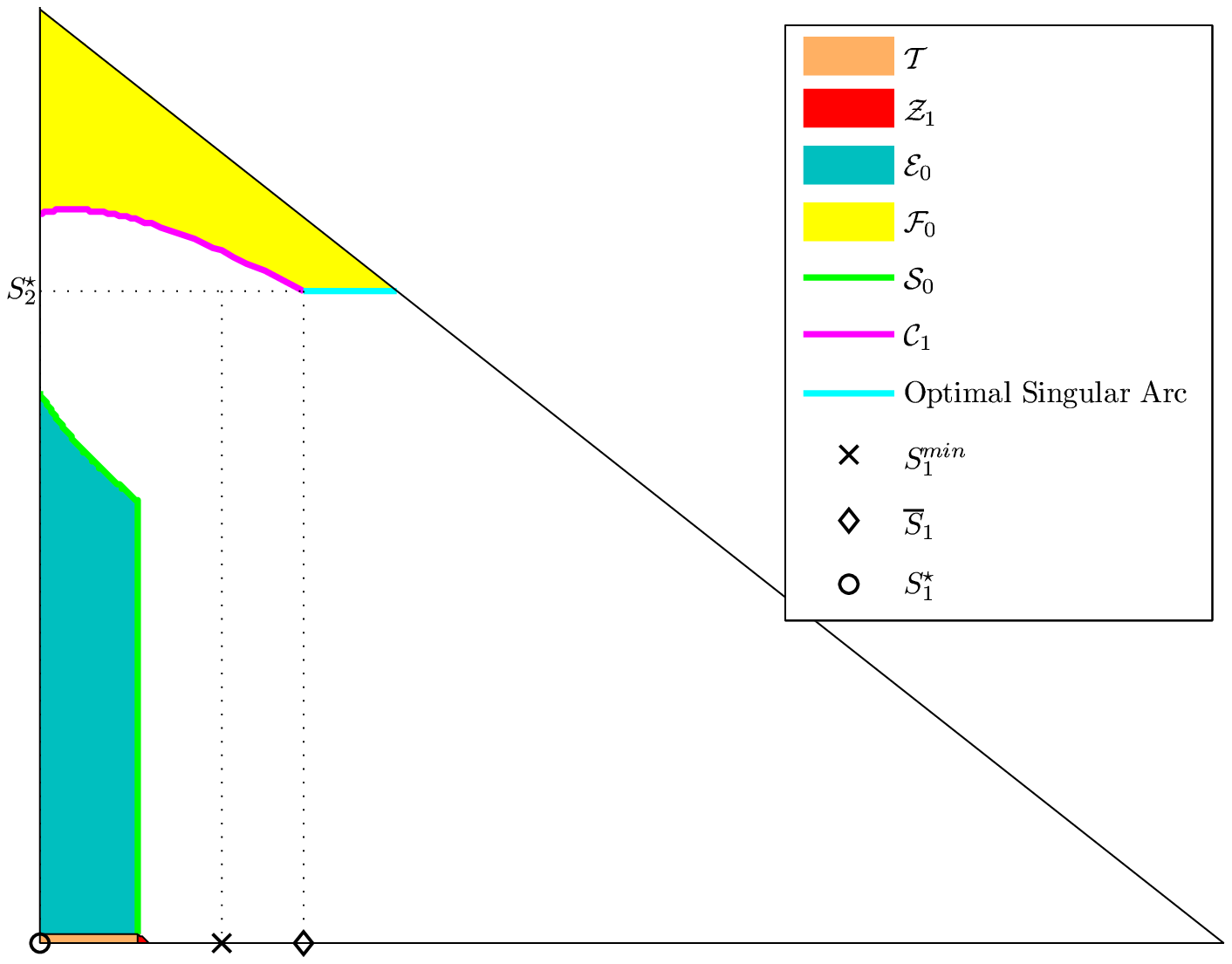} \hspace{4mm}
\includegraphics[scale=0.35]{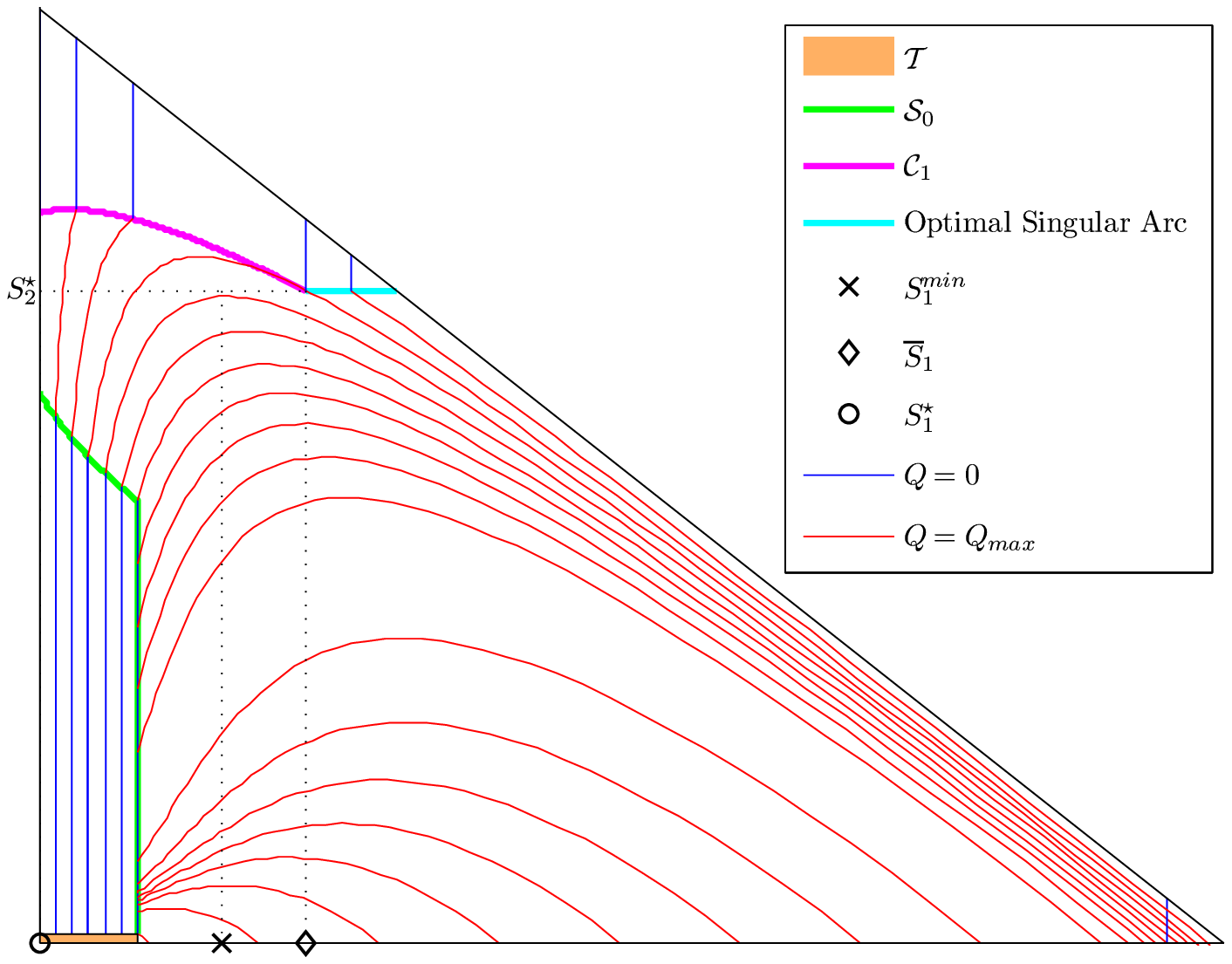}
\caption{Case IIIa. {\it{Picture left}}: Partition of the state space. {\it{Picture right}}: Optimal synthesis provided by Theorem \ref{main1bis}).} \label{fig3a}
\end{center}
\end{figure}

\begin{figure}
\begin{center}
\includegraphics[scale=0.35]{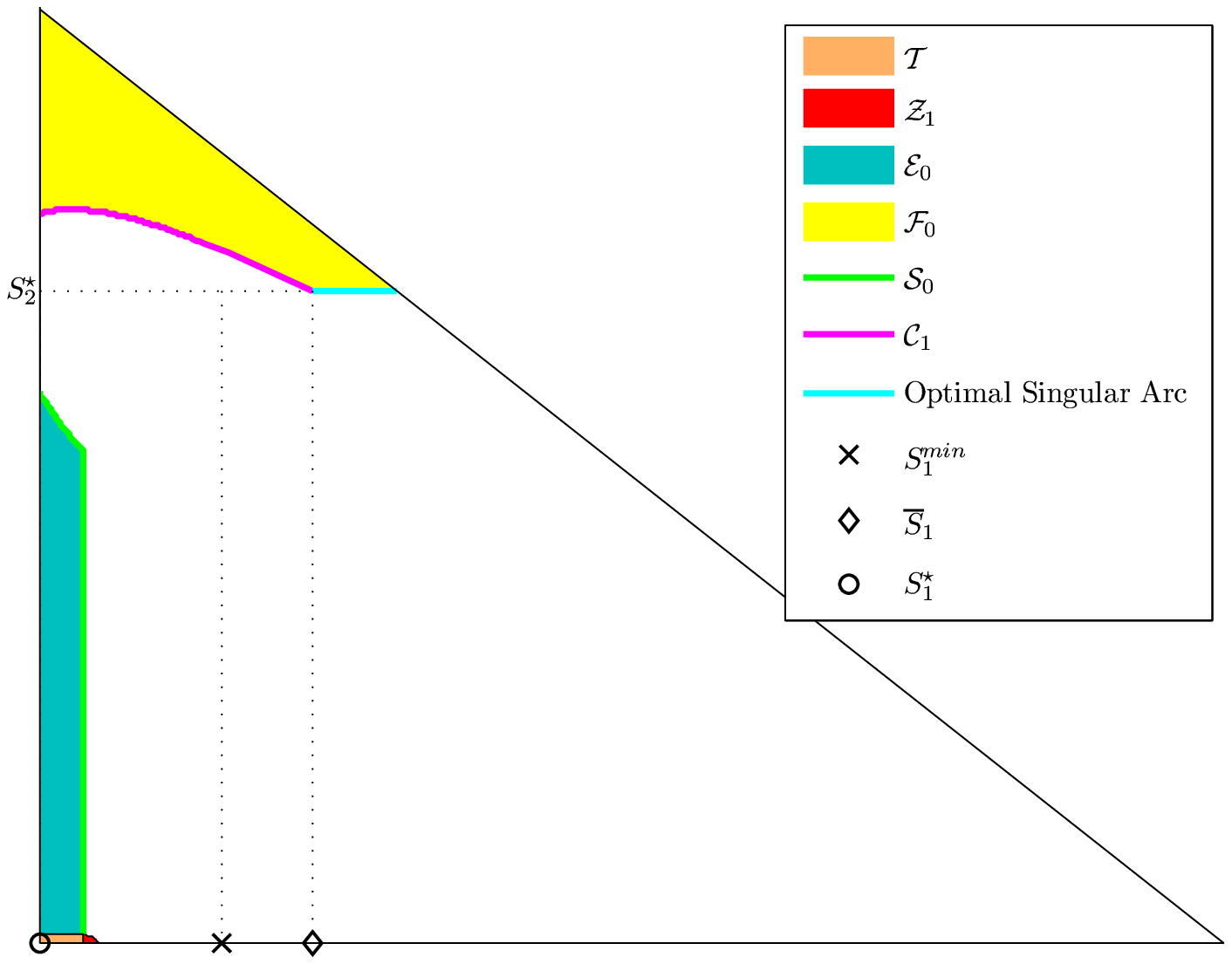} \hspace{4mm}
\includegraphics[scale=0.35]{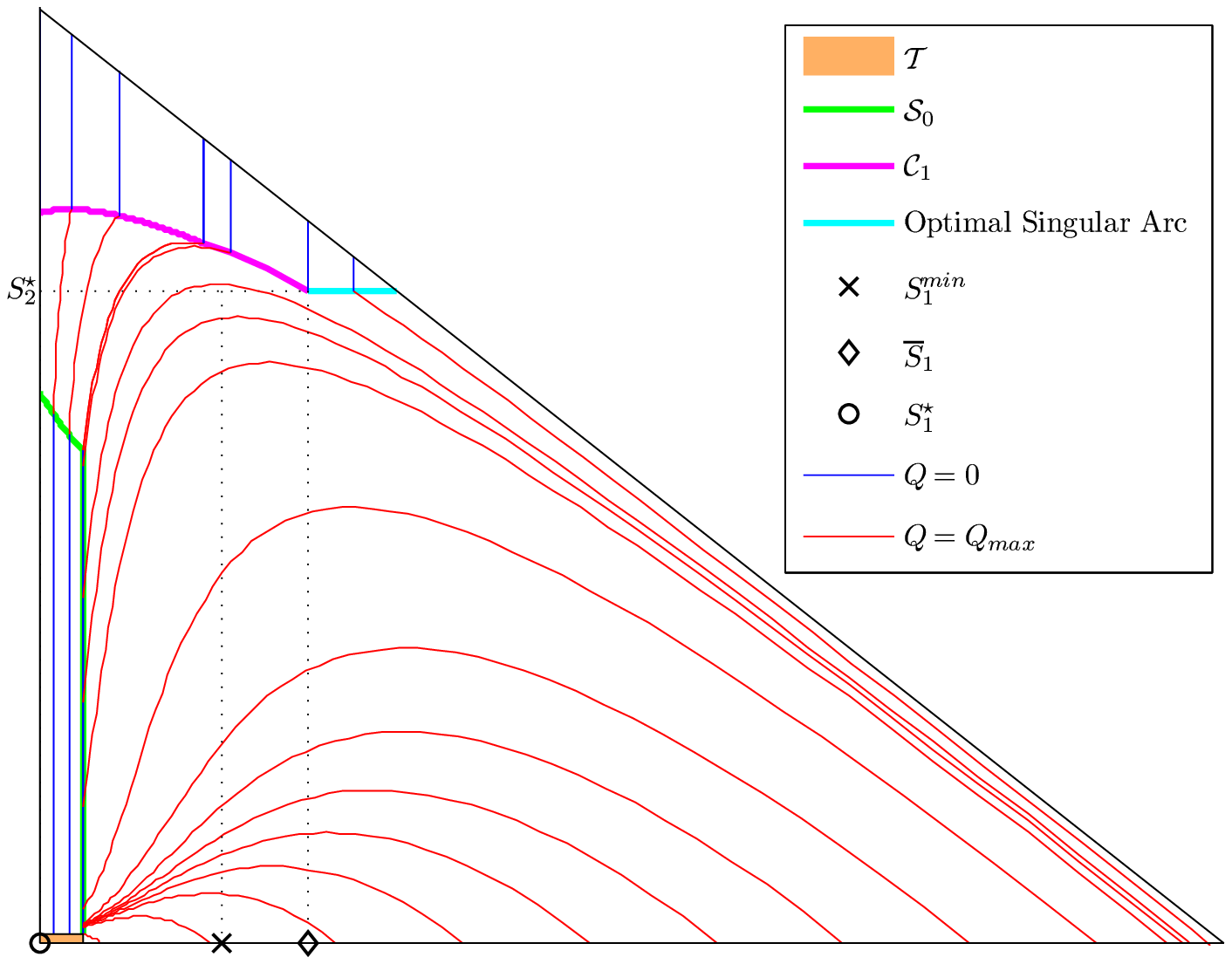}
\caption{Case IIIb. {\it{Picture left}}: Partition of the state space. {\it{Picture right}}: Optimal synthesis provided by Theorem \ref{main1bis}). \label{fig3b}}
\end{center}
\end{figure}

\begin{figure}
\begin{center}
\includegraphics[scale=0.35]{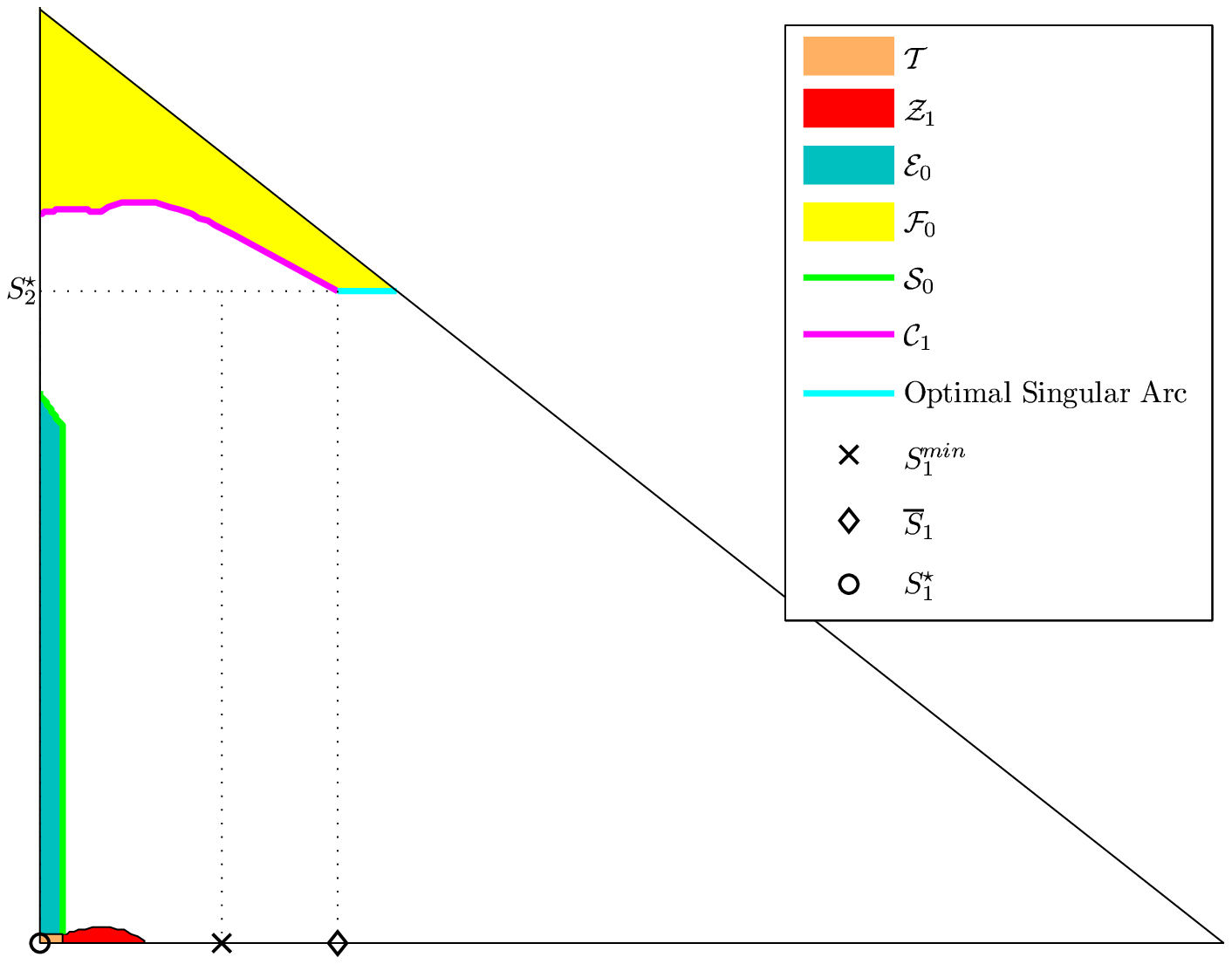} \hspace{4mm}
\includegraphics[scale=0.35]{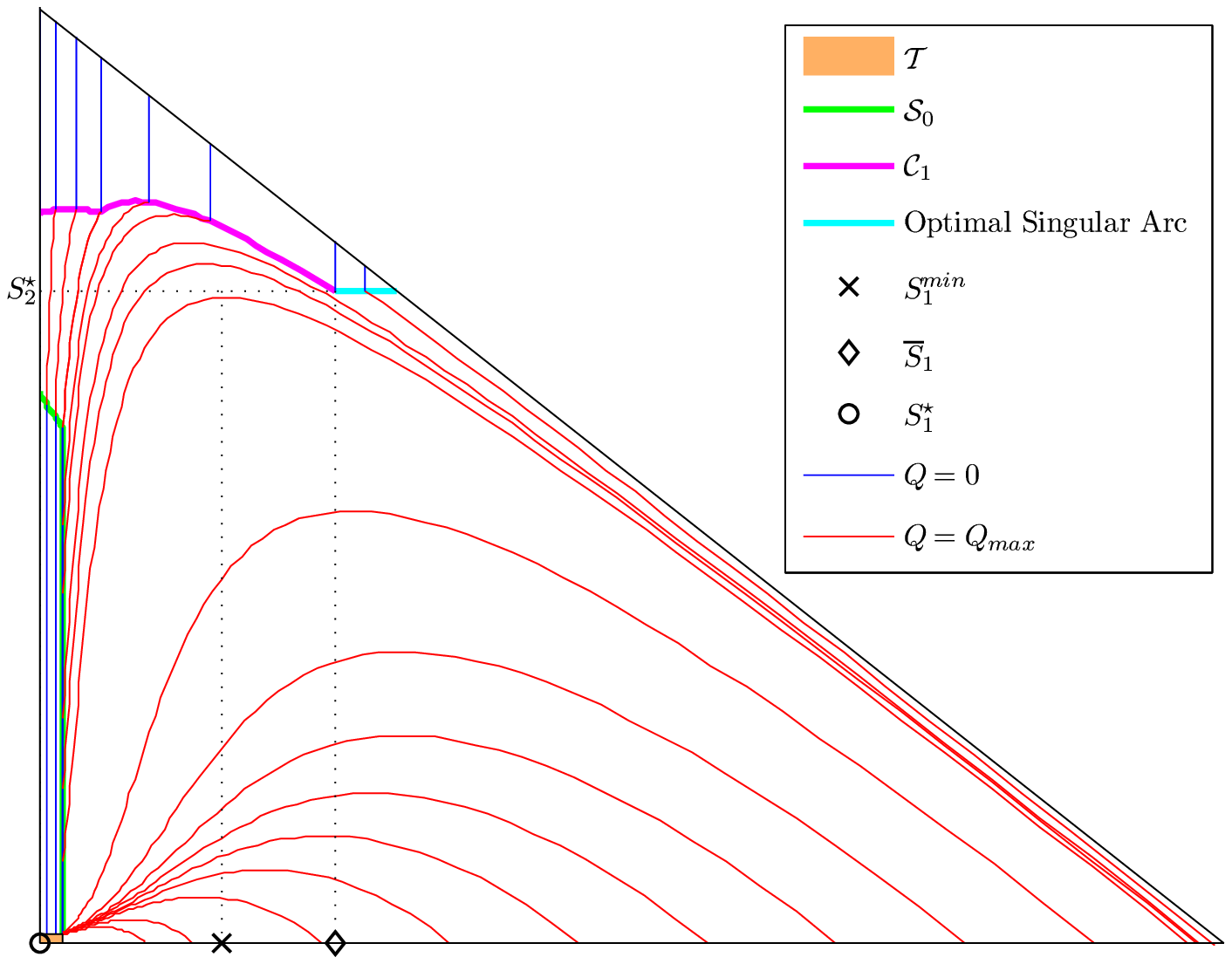}
\caption{Case IIIc. {\it{Picture left}}: Partition of the state space. {\it{Picture right}}: Optimal synthesis provided by Theorem \ref{main1bis}). \label{fig3c}}
\end{center}
\end{figure}

\begin{figure}
\begin{center}
\includegraphics[scale=0.35]{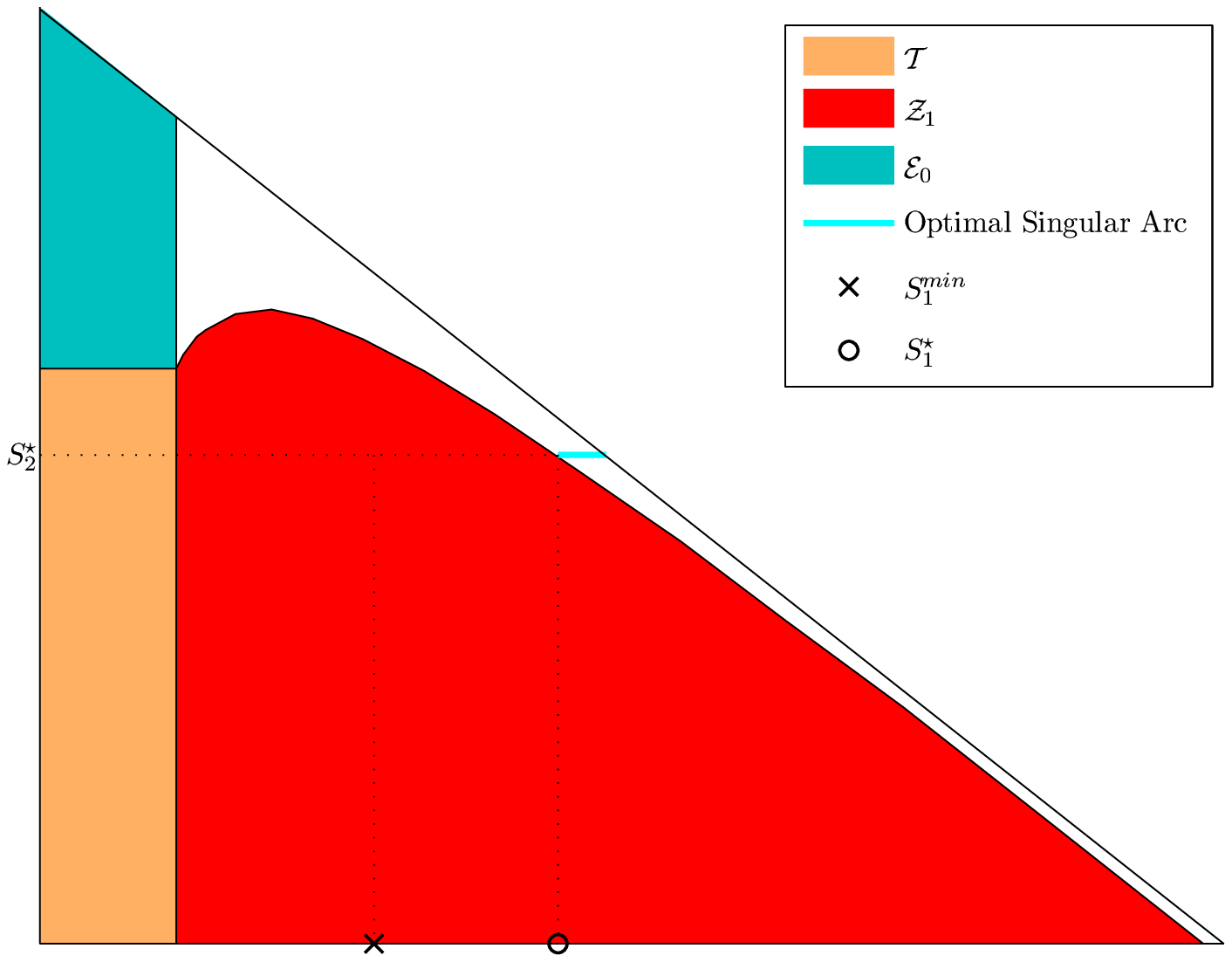} \hspace{4mm}
\includegraphics[scale=0.35]{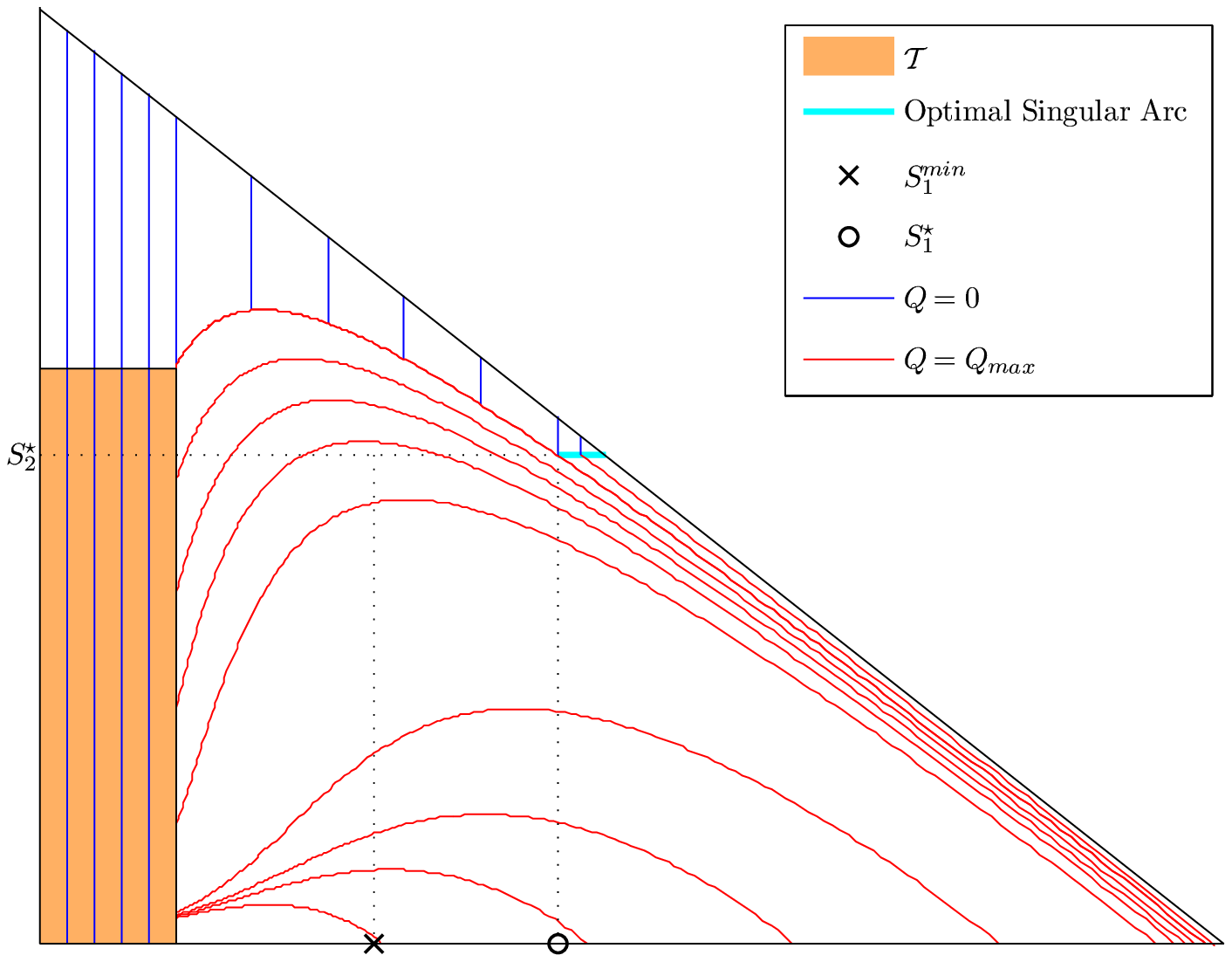}
\caption{Case IVa. {\it{Picture left}}: Partition of the state space. {\it{Picture right}}: Optimal synthesis provided by Proposition \ref{propositionHaldane1}.  \label{fig4a}}
\end{center}
\end{figure}

\begin{figure}
\begin{center}
\includegraphics[scale=0.35]{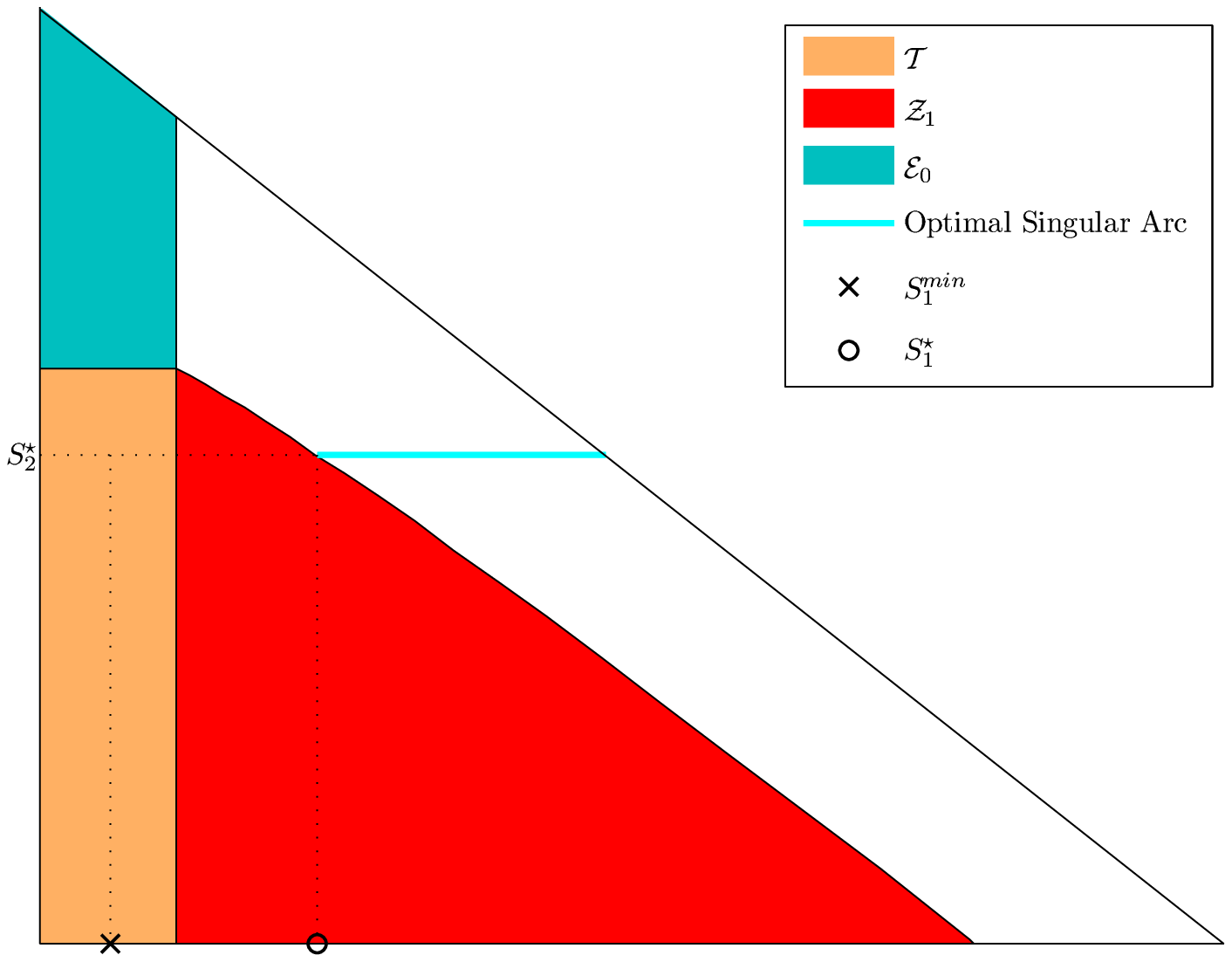} \hspace{4mm}
\includegraphics[scale=0.35]{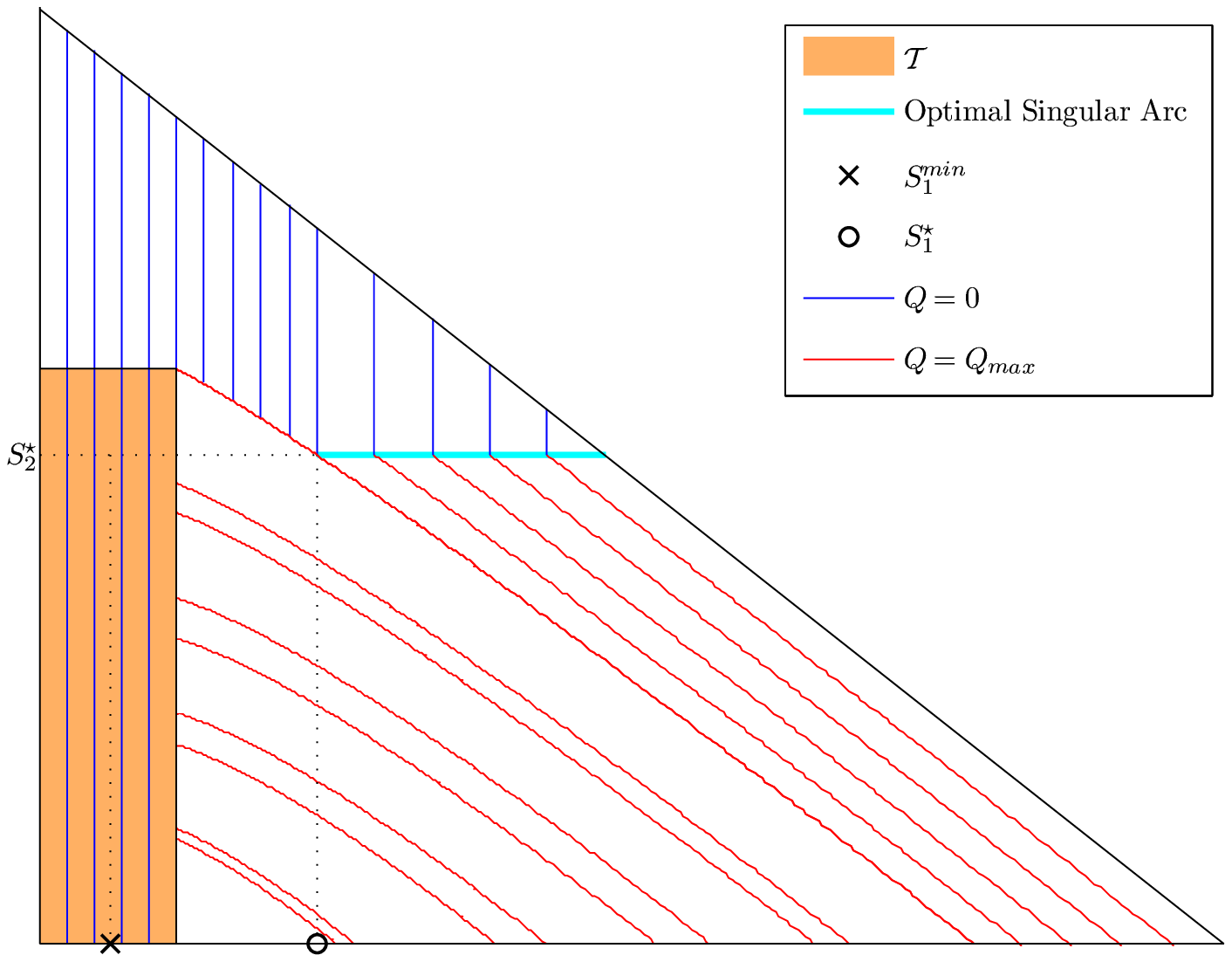}
\caption{Case IVb. {\it{Picture left}}: Partition of the state space. {\it{Picture right}}: Optimal synthesis provided by Proposition \ref{propositionHaldane1}. \label{fig4b}}
\end{center}
\end{figure}

\bigskip

To summarize the optimal synthesis of the problem, we have proceeded as follows. First, we have defined a switching curve 
$\mathcal{C}_0$ as a set of points where the control $u=0$ is optimal until reaching the target. This allows us to 
define an extended target set $\mathcal{E}_0$. Whenever the singular arc is admissible until $\mathcal{E}_0$ the optimal strategy is 
a most rapid approach to the singular arc \cite{BGM13}. In presence of the saturating phenomena, i.e. when the singular arc has a barrier in $\mathcal{D}\backslash \mathcal{E}_0$, then optimal trajectories can have an additional switching point on a curve $\mathcal{C}_1$ 
that can be constructed backward in time from $\mathcal{C}_0$. 
We have pointed out that the difficulty of showing the existence of the switching curve $\mathcal{C}_1$   
whenever the extended target set $\mathcal{E}_0$ does not intersect the singular arc. 
The study of this point is out of the scope of the paper and could deserve further investigations.

The structure of an optimal control is as follows.
We denote by $B_{\pm}$ an arc Bang $u=0$ or $u=1$ and by $S$ a singular arc on a time interval $[t_1,t_2]$. 
We see that when the singular arc is always admissible (see Proposition \ref{propositionHaldane1}), then the optimal synthesis is of type $B_{\pm}SB_{\pm}$ or $B_{\pm}B_{\mp}$, see Fig. \ref{fig4a} and \ref{fig4b}. Hence, optimal trajectories have at most two switching points depending on the initial condition. 
In presence of the saturating phenomena, then the optimal synthesis is of type $B_{\pm}B_{\mp}$, $B_{\pm}SB_{\pm}$, or
$B_{\pm}SB_{\pm}B_{\mp}$. In this case, the optimal synthesis is more intricate and optimal trajectories can have 
three switching points depending on the initial condition. 

Extremal trajectories corresponding to the feedback control law provided by Proposition \ref{propositionHaldane1} are unique. In fact, the uniqueness is clear in the set $Z_1\cup \mathcal{E}_0$, and we can conclude by Green's Theorem (see \cite{bosc}) in $\mathcal{D} \backslash (Z_1\cup \mathcal{E}_0)$. We believe that this property still holds (by exclusion of extremal trajectories that are not optimal) in the case of the feedback law \eqref{feedback1}. 
Finally, we can prove that the value function is continuous
\cite{BC97} (Proposition 1.6 p.230).

\section{Conclusion}

In this work, we have provided a complete analysis of the optimal
synthesis of a model of landfill controlled by the re-circulation
flow. Although the proposed model is simple, the geometry of the
optimal trajectories, depending on the position of the initial
condition with respect to sub-domains that we have characterized, can be
intricate. This analysis can provide useful information in decision
making for the practitioners in different situations, depending on the
characteristics of the landfill (bacterial growth rate and maximum
re-circulation flow).
\begin{itemize}
\item When the landfill operation can be performed in its early stage,
  one may expect to have initial concentration of unsolubilized
  substrate high and solubilized one low. Then, the determination of
  the subset ${\cal Z}_{1}$ appears to be crucial. If it is large, it
  is likely to contain the initial condition and the
  optimal strategy is straightforward: recirculate at the maximal
  speed until the unsolubilized substrate reaches the desired
  concentration. No measurement of the solubilized and no switch on
  the control are necessary, as the state is expected to stay ${\cal
    Z}_{1}$.
\item When the state of the landfill is out of the set ${\cal Z}_{1}$,
  this means that the concentration of solubilized substrate has to
  take large values, and that practitioners would have to stop the
  re-circulation at a certain stage and wait for the solubilized to
  decrease due to the microbial activity.
\item The determination of the best time to stop the re-circulation is
  not necessarily the one when the unsolubilized substrate has reached
  the desired threshold. It can be more efficient to carry on the
  re-circulation until reaching the switching curve ${\cal C}_{0}$.
\item As the concentration of solubilized substrate can significantly
  increase during the transient, its bacterial degradation could
  suffer from an inhibition of the micro-organisms, that is typically
  modeled by a non-monotonic growth rate function, that reaches it
  maximum for some $S_{2}^{\star}$ value. Then, a singular
  arc could be part of the optimal synthesis, which consists in controlling the
  re-circulation flow to regulate the level of the concentration of the
  solubilized substrate at $S_{2}^{\star}$, when its has reached this
  value, until the state reaches the set ${\cal Z}_{1}$ or the
  switching curve ${\cal C}_{0}$.
\item In certain circumstances, the maximal re-circulation flow does not allow to
  maintain the concentration of the solubilized substrate at
  $S_{2}^{\star}$ while reaching the
  switching curve ${\cal C}_{0}$. Then, the optimal decision is to
  anticipate this lack of controllability, and to use the maximal
  re-circulation flow when the state is reaching another switching
  curve ${\cal C}_{1}$.
\item In any case when the state does not reach the set ${\cal
    Z}_{1}$, the final stage is to stop the re-circulation and to
  measure the concentration of solubilized substrate until it reaches
  the desired threshold.
\end{itemize}
Further investigations could concern optimal criteria that take into
consideration the energy spent for the re-circulation and the
valorization of the bio-gas produced by the bacterial activity. 
The consideration of spatial inhomogeneity in the model and its impact
on the optimal strategy could be also the matter of a future research.

\section*{Acknowledgments}
This work was developed in the context of the DYMECOS INRIA
Associated Team, project BIONATURE of CIRIC, INRIA Chile and
CONICYT grant REDES 130067.
The authors express their acknowledgments to F. Carrera, P. Gajardo,
J. Harmand, H. Ramirez, G. Ruiz and V. Riquelme for fruitful exchanges.

\end{document}